\documentclass[12pt,a4paper]{article}

\usepackage[utf8]{inputenc}
\usepackage[english]{babel}
\usepackage{amsmath}
\usepackage{amsfonts}
\usepackage{amssymb}
\usepackage{graphicx}

\usepackage{amsthm}
\usepackage{mathrsfs}
\usepackage{proof}

\usepackage{url}

\usepackage[left=3cm,right=2cm,top=3cm,bottom=2cm]{geometry}
\author{Walter Carnielli and  Marcelo E. Coniglio \\
\\ Centre for Logic,
Epistemology and the History of Science- CLE \\  and Department of Philosophy\\  University of
Campinas -Unicamp\\
walterac@unicamp.br; coniglio@unicamp.br}
\title{Twist-Valued Models for Three-valued Paraconsistent Set Theory}
\date{}

\theoremstyle{theorem}
\newtheorem{teo}{Theorem}[section]
\newtheorem{lema}[teo]{Lemma}
\newtheorem{prop}[teo]{Proposition}
\newtheorem{coro}[teo]{Corollary}

\theoremstyle{definition}
\newtheorem{defi}[teo]{Definition}
\newtheorem{obs}[teo]{Remark}
\newtheorem{Not}[teo]{\textbf{Notation}}

\newcommand{\mbc}{\textbf{mbC}}
\newcommand{\mci}{\textbf{mCi}}
\newcommand{\qmbc}{\textbf{QmbC}}

\newcommand{\lfis}{{\bf LFI}s}
\newcommand{\lfi}{{\bf LFI}}
\newcommand{\cplp}{\text{\bf CPL$^+$}}
\newcommand{\cpl}{\text{\bf CPL}}

\newcommand{\mptz}{{\bf MPT0}}
\newcommand{\lptz}{{\bf LPT0}}
\newcommand{\qlptz}{{\bf QLPT0}}
\newcommand{\aptz}{\ensuremath{\mA_{PT0}}}

\newcommand{\lfium}{{\bf LFI1}}
\newcommand{\qlfium}{{\bf QLFI1}}
\newcommand{\dacdot}{{\bf J3}}

\newcommand{\lpt}{{\bf LPT}}
\newcommand{\mpt}{{\bf MPT}}

\newcommand{\ZFmbc}{{\bf ZFmbC}}
\newcommand{\ZFlp}{\ensuremath{{\bf ZF}_{\lptz}}}
\newcommand{\ZF}{{\bf ZF}}
\newcommand{\ZFmci}{{\bf ZFmCi}}

\newcommand{\ZFC}{{\bf ZFC}}
\newcommand{\pst}{\ensuremath{\mathbb{PS}_3}}
\newcommand{\pert}{\ensuremath{\,\epsilon\,}}

\newcommand{\termvalue}[1] {\lbrack\!\lbrack #1 \rbrack\!\rbrack}

\newcommand{\wneg}{\ensuremath{\lnot}}
\newcommand{\sneg}{\ensuremath{{\sim}}}
\newcommand{\imp}{\rightarrow}
\newcommand{\Ra}{\Rightarrow}

\newcommand{\sse}{\leftrightarrow}

\newcommand{\kax}{\textbf{Ax1}}
\newcommand{\axTrans}{\textbf{Ax2}}
\newcommand{\axed}{\textbf{Ax3}}
\newcommand{\axeea}{\textbf{Ax4}}
\newcommand{\axeeb}{\textbf{Ax5}}
\newcommand{\axouda}{\textbf{Ax6}}
\newcommand{\axoudb}{\textbf{Ax7}}
\newcommand{\axoue}{\textbf{Ax8}}
\newcommand{\axouimp}{\textbf{Ax9}}
\newcommand{\axtnd}{\ensuremath{\textbf{TND}_\neg}}
\newcommand{\axtndcl}{\textbf{TND}}
\newcommand{\axexpl}{\textbf{exp}}

\newcommand{\axdneg}{\textbf{dneg}}
\newcommand{\axnegou}{\ensuremath{\textbf{neg}\lor}}
\newcommand{\axnege}{\ensuremath{\textbf{neg}\land}}
\newcommand{\axnegimp}{\ensuremath{\textbf{neg}\imp}}
\newcommand{\MP}{\textbf{MP}}

\newcommand{\cons}{\ensuremath{{\circ}}}

\newcommand{\mA}{\ensuremath{\mathcal{A}}}

\newcommand{\matM}{\ensuremath{\mathcal{M}}}

\newcommand{\tA}{\ensuremath{T_\mathcal{A}}}
\newcommand{\tmA}{\ensuremath{\mathcal{T}_\mathcal{A}}}
\newcommand{\matA}{\ensuremath{\mathcal{MT}_\mathcal{A}}}
\newcommand{\tmpA}{\ensuremath{\mathcal{T}_{\mathcal{A}'}}}
\newcommand{\tmzA}{\ensuremath{\mathcal{T}_{\mathcal{A}^\ast}}}

\newcommand{\sent}{$Sen(\Theta)$}
\newcommand{\tert}{$Ter(\Theta)$}
\newcommand{\fort}{$For(\Theta)$}
\newcommand{\ctert}{$CTer(\Theta)$}

\begin{document}

\maketitle

\begin{abstract}
Boolean-valued models  of  set  theory were introduced by  Scott  and  Solovay in 1965  (and independently by Vop\v{e}nka  in the same year),  offering a natural and rich alternative for  describing forcing. The original method was adapted by Takeuti, Titani, Kozawa and Ozawa to lattice-valued models of set theory. After this,  L\"owe and Tarafder proposed a class of algebras based on  a certain kind of implication which satisfy several axioms of \ZF. From this class, they found a specific three-valued model called \pst\ which satisfies all the axioms of \ZF, and can be expanded with a paraconsistent negation *,  thus  obtaining  a paraconsistent model of \ZF. We observe here that (\pst,*)  coincides (up to language) with da Costa and D'Ottaviano logic \dacdot, a three-valued paraconsistent logic that have been proposed independently in the literature by several authors and with different motivations: for instance, it was  reintroduced as {\bf CLuNs}, \lfium\ and \mpt, among others.

We propose in this paper a family of algebraic models of \ZFC\ based on a paraconsistent three-valued logic called \lptz, another linguistic variant of \dacdot\ and so of (\pst,*) introduced by us in 2016. The semantics of \lptz, as well as of its first-order version \qlptz, is given by twist structures defined over arbitrary complete Boolean agebras. From this, it is possible to adapt the standard Boolean-valued models of  (classical) \ZFC\  to an expansion of \ZFC\ by adding a paraconsistent negation. This paraconsistent set theory is based on \qlptz, hence it is a  paraconsistent expansion of \ZFC\  characterized by a class of twist-valued models.

We argue that the implication operator of \lptz\   considered in this paper  is, in a sense, more suitable for a paraconsistent set theory than the implication of \pst:  indeed, our implication  allows for  genuinely inconsistent sets (in  a  precise sense,  $\termvalue{(w \approx w)}={\bf \frac{1}{2}}$ for some $w$). It is to be remarked  that our implication  does not   fall under the definition of the  so-called `reasonable implication algebras' of L\"{o}we and Tarafder. This suggests   that `reasonable implication algebras'  are just one way to define a paraconsistent set theory,  perhaps not the most appropriate.

Our twist-valued models for \lptz\ can be easily adapted  to provide  twist-valued models for  (\pst,*); in this way twist-valued  models generalize L\"{o}we and Tarafder's three-valued \ZF\ model, showing that all of them (including (\pst,*)) are, in fact, models of \ZFC\ (not only of \ZF). This offers more options for investigating  independence  results in paraconsistent set theory.
\end{abstract} 

 \section{On models of set theory:  G\"odel shrinks, Cohen expands}

The interest for -- and the  overall  knowledge about --  models for set theory changed dramatically after the famous invention (or discovery) of  Paul Cohen's  methods of forcing.  Cohen was able to show that  the notion of cardinal number is elastic and relative, in contrast with the methods  of  ``inner  models''  that   G\"odel  used.   G\"odel has shown that, by  shrinking the totality of sets in a  model, they would turn to be `well-behaved'. As a consequence, the  constructible  sets could  not  be  used  to prove the relative  consistency  of  the negation of the  Axiom of Choice (AC) or of the Continuum Hypothesis (CH). 
Paul J. Cohen, on the contrary,  had the idea of reverting the paradigm, and instead of cutting down the sets within  models, found a way to expand a countable standard  model into a standard model in which CH or AC can be  false, doing this  in a  minimalist but  controlled  fashion.  Cohen elements  are  `bad-behaved', but finely guided so  as to make `logical space' for the independence of  AC and   CH,

As Dana Scott puts  in the  forward of Bell's book~\cite{bell}, 
``Cohen's achievement lies in being able to expand models (countable, standard  models) by adding new sets in a very economical fashion: they more or less have only the properties they are forced to have by the axioms (or by the truths of the given model).''
Cohen's methods, however, are  not easy, being  regarded by some researchers as somewhat lengthy  and tedious -- but were the only tool available  until the  Boolean-valued  models   of  set  theory   put  forward  by  Scott  and  Solovay   (and independently by Vop\v{e}nka) in 1965  offered a     more natural and rich alternative for  describing forcing.   This does not discredit the brilliant  idea of Cohen, who  did not have the machinery of Boolean-valued  models available at his time. 

What is a  Boolean-valued model?  The intuitive idea is to pick a suitable Boolean algebra  \mA, and  define  the set of all \mA-valued sets in M, generalizing the  familiar  $\{0,1\}$   valued  models.  Then add to the language  one constant  symbol for each element of the model.  After this, define a  map $\varphi \mapsto \termvalue{\varphi}^{\mA\ }$  from the sentences in S to  \mA\  which   obey  certain equations so that it should  assign  1 to all the axioms of \ZFC.

The  resulting  structure $M_\mA$  will not   be a standard  model of \ZFC, because it will consist of 
``relaxed sets''  somehow similar to fuzzy sets,  and  not sets properly. If we take  an arbitrary sentence about  sets (for instance, ``does  $Y$  is a  member  of  $X$'' ?) and ask whether it holds in  $M_\mA$, then the answer may  be neither plain ``yes'' nor ``no'',  but some element  of the Boolean algebra  \mA\ meaning the ``degree''  to  which  $Y$ is a member of $X$.  However,    $M_\mA$   will satisfy \ZFC, and to turn   
$M_\mA$  into an actual model of \ZFC\ with certain desired properties  it is sufficient to take a suitable quotient of  $M_B$  that eliminates the elements  of fuzziness. 

Boolean-valued models not only avoid tedious details of Cohen's original construction, but permit a great generalization by varying on any Boolean algebra.

\section{Losing unnecessary weight: the role of alternative set theories}

It is a well-known historical fact  that the discovery of the paradoxes in set theory  and in the foundations of mathematics was the fuse that fired the revolution  in contemporary set theory around its efforts to attempt to rescue Cantor's naive theory from triviality. The usual culprit was  the Principle of (unrestricted) Abstraction, also known as the Principle of Comprehension. Unrestricted abstraction  allows sets to be defined by arbitrary conditions, and this freedom combined with  the axiom of extensionality,  leads to a contradiction, which by its turn leads to triviality in the sense that ``everything goes'', when   the laws of the underlying logic obey the standard principles that comprise the so-called ``classical'' logic.

But there is a way out from this maze. Paraconsistent set theory is the theoretical move to  maintain the freedom of defining sets,   while  stripping the theory of unnecessary principles so as to avoid triviality, a  disastrous consequences of contradictions involving sets in \ZF.

This philosophical maneuver is in frank opposition to  traditional strategies, which deprive the  freedom of set theory so appreciated by Cantor, by  maintaining  the underlying logic and weakening  the Principle of Abstraction,
  
An analogy may be instructive. The basic goal of reverse mathematics is to study  the relative logical strengths of theorems from ordinary non-set theoretic mathematics. To this end, one tries to find the minimal natural axiom system $A$ that  is capable of proving a theorem $T$.
 
In a  perhaps vague, but  illuminating analogy,  paraconsistent logic tries to find the minimal natural principles that are capable of permitting us to reason in generic circumstances, even in the  undesired circumstances of contradictions. 
 
This does  not mean that  contradictions are necessarily real:  \cite{car-rod:2019} gives a  formal system and a corresponding intended interpretation, according to which true contradictions are not tolerated. Contradictions are, instead, epistemically understood as conflicting evidence. There  are indeed  many cases of contradictions in reasoning, but  the classical principle \emph{Ex Contradictione Quodlibet}, or Principle of  Explosion, is  neither used in mathematics in general;  it  is not, therefore, a characteristic of good reasoning, and  has to be abandoned.

Some people may be mislead by thinking that  \emph{Reductio ad Absurdum}, which  is  a useful and robust    rule of inference, would be lost  by abandoning the Principle of  Explosion.
This is not so: even if discarding such a principle,  proofs  by \emph{Reductio ad Absurdum}   get unaffected, as long as one  can define a strong negation. This is achieved in many  paraconsistent  logics, in particular  in all the logics of the family of the  Logics of Formal Inconsistency  (\lfis), see~\cite{CM02,CCM,CC16}. Reasoning does not necessarily  require  the full power  of  \emph{Ex Contradictione Quodlibet}, because   contradictions reached in a \emph{Reductio}  proof are not  really used to cause any deductive explosion;  what is used is the manipulation of negation.

\section{Expanding Cohen's expansion: twist-valued\\ models} \label{Intro-twist}


Boolean-valued models  were adapted by Takeuti, Titani, Kozawa and Ozawa to lattice-valued models of set theory, with applications to quantum set theory and fuzzy set theory (see~\cite{tak:tit:92,tit:99,tit:koz:03,Oza:07,Oza:17}). The guidelines of these constructions were taken by  L\"owe and Tarafder in~\cite{low:tar:15} in order to obtain a  three-valued model (in the form of a  lattice-valued model)  for a paraconsistent set theory based  on \ZF. They propose a class of algebras based on  a certain kind of implication, called {\em reasonable implication algebras} (see Section~\ref{genPS3}) which satisfy several axioms of \ZF. From this class, they found an especific three-valued model which satisfies all the axioms of \ZF, and it can be expanded to an algebra  (\pst, *) with a paraconsistent negation *, obtaining so a paraconsistent model of \ZF.  As we discuss  in Section~\ref{genPS3},   the logic (\pst, *) is the same as the logic \mpt\  introduced in~\cite{ConSil:14}, and coincides up to language with the logic \lptz\ adopted in the present   paper. Here,  we will  introduce the notion of  twist-valued models  for a paraconsistent set theory \ZFlp\    based on \qlptz, a first-order version of \lptz.  Our models,  defined  for any complete Boolean algebra \mA,  constitute a  generalization of the Boolean-valued models for set theory, at the same time generalizing 
L\"owe and Tarafder's three-valued model. Indeed, in Section~\ref{genPS3} the model of \ZF\ based on (\pst, *) will be generalized to twist-valued models over an arbitrary complete Boolean algebra, obtaining so a class of models of \ZFC. The structure over (\pst, *)  will constitute a particular case, by considering the two-element complete Boolean algebra. As a consequence of this, it follows that L\"owe and Tarafder's three-valued structure is, indeed,  a model of \ZFC.

Twist-structure  semantics have been independently proposed by  M. Fidel~\cite{fid:78}  and D. Vakarelov~\cite{vaka:77}, in order to semantically characterize  the well-known Nelson logic. A twist structure  consists of operations  defined on the  cartesian product of the universe  of a lattice,   $L\times L $  so that the negative and positive algebraic characteristics  can be treated separately. In terms of logic, a pair $(a,b)$ in $L \times L$ is such that $a$ represents a truth-value for a formula $\varphi$ while $b$ corresponds to a truth-value for the negation of $\varphi$. That is, $a$ is a positive value for $\varphi$ while $b$ is a negative value for it, thus justifying the name `twist structures' given for this kind of algebras.  This strategy is especially useful for obtaining semantical characterizations for non-standard logics. As a limiting case, a Boolean algebra  turns  out being  a particular case  of  twist structures  when there  is no need to give separate attention to negative and positive algebraic characteristics, since the latter are uniquely obtained from the former by the dualizing Boolean complement $\sneg$. In this case,  every pair $(a,b)$ is of the form $(a, \sneg a)$, hence the second coordinate is redundant. Our proposal is based on models for \ZF\ based on twist structures, thus the sentences of the language of \ZF\ will be interpreted as pairs $(a,b)$ in a suitable twist structure, such that the supremum $a \vee b$ is always $1$, but the infimum $a \wedge b$ is not necessarily equal to $0$. This corresponds to the validity of the third-excluded middle for the non-classical negation of the underlying logic, while the explosion law $\varphi \wedge \neg \varphi \to \psi$ is not valid in general in the underlying paraconsistent logic \lptz. A somewhat related approach was proposed by Libert in~\cite{Li:05}: he proposes models for a naive set theory in which the truth-values are pairs of sets $(A,B)$ of a universe $U$ such that $A \cup B=U$ where $A$ and $B$ represent, respectively, the extension and the anti-extension of a set $a$. However, besides this  similarity, our approach is quite different: we are interesting in giving paraconsistent models for \ZFC\ and not in new models for  Naive set theory.

It is important to notice that there exists in the literature several approaches to paraconsistent set theory, under different perspectives. In particular, we propose in~\cite{CC13} a paraconsistent set theory based on several \lfis, but that approach differs from the one in the present paper. First, in the previous paper the systems were presented axiomatically, by means of suitable   modifications of \ZF. Moreover, in that logics  a consistency predicate $C(x)$ was considering, with the intuitive meaning that `$x$ is a consistent set'. On the other hand, in the present paper a model for standard \ZFC\ will be presented instead of a Hilbert calculus for a modified version of \ZF.  We will return to this point in Section~\ref{sect-ZFparacon}.

As mentioned above, twist structures over a Boolean algebra generalize Boolean  algebras,  and are by their turn generalized  by the \emph{swap structures } introduced in~\cite[Chapter~6]{CC16} (a previous notion of swap structures  was given in~\cite{CC-swap}).  Swap structures are non-deterministic algebras defined over the three-fold Cartesian product $\mA\times \mA\times \mA$ of a given Boolean algebra so that in a triple $(a, b, c)$  the first component $a$  represents the truth-value of a given formula $\varphi$  while $b$  and  $c$ represent,   respectively,  possible values for the paraconsistent negation $\neg\varphi$ of $\varphi$,  and  for the consistency   $\circ\varphi$ of  $\varphi$.

Swap structures are  committed to semantics  with a   non deterministic character, while twist structures  are used when the semantics are deterministic (or truth-functional).  Definition~\ref{defKlfi1}  below  shows  how  the definition of twist structures for the three-valued logic   $\lfium_\circ$ introduced in~\cite[Definition~9.2]{CFG18} can be adapted to \lptz.

As noted in  Section~\ref{twistm}, the three-valued logic (\pst, *)  used in~\cite{tar:15}  already  appears   in~\cite{ConSil:14} under the name \mpt, and it is equivalent to \lptz\ and also to   $\lfium_\circ$. Variants of this  logic have been independently proposed by different authors at  with different motivations in several occasions (for instance, as the well-known da Costa and D'Ottaviano's logic \dacdot). The naturalness of this logic is reflected by the fact that the three-valued algebra of \lptz\ (see Definition~\ref{MPT0-def} below) is equivalent, up to language, to the algebra underlying   \L ukasiewicz  three-valued logic \L3. The only difference is that in the former the set of distinguished (or designated)  truth values is  $\{1, \frac{1}{2} \}$ instead of  $\{1\}$, and this is why \lptz\ is paraconsistent while \L3 is paracomplete. 

Twist-valued models work beautifully as  enjoying  many properties   similar  to  Boolean-valued models (when restricted to pure \ZF-languages).  Such  similarities  lead to a natural proof  that \ZFC\ is valid w.r.t. twist-valued models, as  our central Theorem~\ref{modelZFC}  shows. 
This paper deals with  a paraconsistent set theory  named  \ZFlp, defined by using as the underlying logic  a  first-order version of \lptz, called \qlptz,   proposed in~\cite{CFG19} under the form of $\qlfium_\circ$ (that is, by replacing the strong negation $\sneg$ by the consistency operator~$\circ$).  

The paraconsistent character of twist-valued models as regarding  \ZFlp\ as  rival  of \ZFC\ is emphasized. Despite having some  limitative results, as  much as  L\"owe and Tarafder's  model,   \ZFlp\ has a great potential as  generator of models for  paraconsistent set theory.  A subtle, but critical advantage of our models is that  the implication operator of  \lptz\ is much more suitable for a paraconsistent set theory than the one of \pst. Indeed, our models  allow for  inconsistent sets, and this is of paramount importance, as we argue  below. Moreover, as pointed out above, our  models  generalize the three-valued model   based on \pst, since they can be defined for any complete Boolean algebra. In this way, we have  several  models at our disposal, and in principle this can be used to investigate  independence results in paraconsistency set  theory.

Albeit  Boolean-valued  models  and their  generalization in the form of  twist-valued models are naturally devoted to  study independence results, this paper does not tackle  this big questions yet. The paper, instead, is dedicated to  clarifying such models  while establishing  their basic properties.

\section{The logic \lptz}

In this section the logic \lptz\ will be briefly discussed, including its twist structures semantics. From now on, if $\Sigma'$ is a propositional signature then, given a denumerable set $\mathcal{V}=\{p_1,p_2,\ldots\}$ of propositional variables, the propositional language generated by $\Sigma'$ from $\mathcal{V}$ will be denoted by ${\cal L}_{\Sigma'}$. The paraconsistent logics considered in this paper  belong to the class of  logics known as {\em logics of formal inconsistency}, introduced in~\cite{CM02}  (see also~\cite{CCM,CC16}).

\begin{defi}
Let  ${\bf L}=\langle  \Sigma',\vdash \rangle$ be a Tarskian, finitary and structural logic defined over a propositional signature $\Sigma'$, which contains a
negation $\neg$, and let  $\circ$ be a (primitive or defined) unary connective. The logic
${\bf L}$ is said to be  a {\em logic of formal inconsistency} (\lfi) with respect to $\neg$ and $\circ$ if the following holds:
\begin{itemize}
       \item [(i)] $\varphi,\neg\varphi\nvdash\psi$ for some $\varphi$ and $\psi$;
       \item [(ii)] there are two formulas $\varphi$ and $\psi$ such that
     \begin{itemize}
       \item [(ii.a)] $\circ\alpha,\varphi \nvdash \psi$;
       \item [(ii.b)] $\circ\alpha, \neg \varphi \nvdash \psi$; 
\end{itemize}
       \item [(iii)]  $\circ\varphi,\varphi,\neg\varphi\vdash \psi$ for every $\varphi$ and $\psi$. 
\end{itemize}
\end{defi}

\

Recall the logic \mptz\ presented in~\cite{CC16} as a linguistic variant of the logic \mpt\  introduced in~\cite{ConSil:14}. 

\begin{defi} (Modified Propositional logic of Pragmatic Truth \mptz, \cite[Definition~4.4.51]{CC16}) \label{MPT0-def}
Let $\matM_{PT0} = \langle M, D\rangle$ be the three-valued logical matrix over $\Sigma=\{\land,\lor,\imp, \sneg, \wneg\}$ with domain $M= \{{\bf 1},{\bf \frac{1}{2}}, {\bf 0}\}$ and set of designated values $D=\{{\bf 1},{\bf \frac{1}{2}}\}$ such that the operators are defined as follows:

\ 

\begin{center}
\begin{tabular}{|c||c|c|c|}
\hline
$\land $ & {\bf 1} & ${\bf \frac12}$ & {\bf 0} \\ \hline \hline
{\bf 1} & {\bf 1} & ${\bf \frac12}$ & {\bf 0} \\ \hline
${\bf \frac12}$ & ${\bf \frac{1}{2}}$ & ${\bf \frac12}$ & {\bf 0} \\ \hline
{\bf 0} & {\bf 0} & {\bf 0} & {\bf 0} \\ \hline
\end{tabular}
\hspace{0.5cm}
 \begin{tabular}{|c||c|c|c|}
\hline
$\lor $ & {\bf 1} & ${\bf \frac12}$ & {\bf 0} \\ \hline \hline
{\bf 1} & {\bf 1} & {\bf 1} & {\bf 1} \\ \hline
${\bf \frac12}$ & {\bf 1} & ${\bf \frac12}$ & ${\bf \frac12}$ \\ \hline
{\bf 0} & {\bf 1} & ${\bf \frac12}$ & {\bf 0} \\ \hline
\end{tabular}
\hspace{0.5cm}
\begin{tabular}{|c||c|c|c|}
\hline
$\imp$ & {\bf 1}  & ${\bf \frac12}$  & {\bf 0} \\
 \hline \hline
     {\bf 1}    & {\bf 1}  & ${\bf \frac12}$  & {\bf 0}   \\ \hline
     ${\bf \frac12}$  & {\bf 1}  & ${\bf \frac12}$  & {\bf 0}   \\ \hline
     {\bf 0}    & {\bf 1}  & {\bf 1}  & {\bf 1}   \\ \hline
\end{tabular}
\end{center}

\ \\

\begin{center}
\begin{tabular}{|c||c|} \hline
$\quad$ & $\sneg$ \\
 \hline \hline
     {\bf 1}   & {\bf 0}    \\ \hline
     ${\bf \frac12}$ & {\bf 0}  \\ \hline
     {\bf 0}   & {\bf 1}    \\ \hline
\end{tabular}
\hspace{0.5cm}
\begin{tabular}{|c||c|} \hline
$\quad$ & $\wneg$ \\
 \hline \hline
     {\bf 1}   & {\bf 0}    \\ \hline
     ${\bf \frac12}$ & ${\bf \frac12}$  \\ \hline
     {\bf 0}   & {\bf 1}    \\ \hline
\end{tabular}
\end{center}

\ \\

\noindent The logic  associated to the logical matrix $\matM_{PT0}$ is called \mptz. The three-valued algebra underlying $\matM_{PT0}$ will be called \aptz.
\end{defi}

\noindent Observe that $x \to y = \sneg x \vee y$ for every $x,y$.
Recall that, by definition, the consequence relation $\vDash_{\mptz}$ of \mptz\ is given as follows: for every $\Gamma\cup\{\varphi\} \subseteq {\cal L}_{\Sigma}$, $\Gamma\vDash_{\mptz} \varphi$ iff, for every homomorphism $v:{\cal L}_{\Sigma} \to  M$ of algebras over $\Sigma$, if $v[\Gamma] \subseteq D$ then $v(\varphi) \in D$.

From~\cite{CC16} a sound and complete Hilbert calculus for \mptz, called \lptz, can be defined. This calculus is an axiomatic extension  of a Hilbert calculus for classical propositional logic \cpl\ over the signature $\Sigma_c=\{\land, \lor, \imp,\sneg\}$. From now on, $\varphi \leftrightarrow \psi$ will be an abbreviation for the formula $(\varphi \to \psi) \wedge (\psi \to \varphi)$. 

\begin{defi} (The calculus \lptz, \cite[Definition~4.4.52]{CC16}) \label{LPT0-def} The Hilbert calculus  \lptz\  over $\Sigma$ is defined as  follows:\footnote{To be rigorous, in~\cite[Theorem~4.4.56]{CC16} an additional axiom schema is required:\ $\wneg\sneg\varphi \rightarrow \varphi$. However, it is easy to prove that this axiom is derivable from the others, by using \MP.}\\
\newpage
{\bf Axiom Schemas:}\\

$\begin{array}{ll}
(\kax) & \varphi \to(\psi \to \varphi)\\[2mm]
(\axTrans) & (\varphi \to (\psi \to \gamma)) \to ((\varphi \to \psi) \to (\varphi \to \gamma ))\\[2mm]
(\axed) & \varphi \to(\psi \to (\varphi \wedge \psi))\\[2mm]
(\axeea) & (\varphi \wedge \psi)\to \varphi\\[2mm]
(\axeeb) & (\varphi \wedge \psi)\to \psi\\[2mm]
(\axouda) & \varphi \to( \varphi \vee \psi)\\[2mm]
(\axoudb) & \psi \to( \varphi \vee \psi)\\[2mm]
(\axoue) & (\varphi \to \gamma)\to (( \psi \to \gamma)\to ((\varphi \vee \psi)\to \gamma))\\[2mm]
(\axouimp) & \varphi \vee (\varphi \to \psi)\\[2mm]
(\axtndcl) & \varphi \vee \sneg \varphi\\[2mm]
(\axexpl) & \varphi \imp \big(\sneg \varphi \imp \psi\big)\\[2mm]
(\axtnd) & \varphi \vee \neg \varphi\\[2mm]
(\axdneg) & \wneg\wneg\varphi \leftrightarrow \varphi\\[2mm]
(\axnegou) & \wneg(\varphi \lor \psi) \leftrightarrow (\wneg \varphi \land \wneg\psi)\\[2mm]
(\axnege) & \wneg(\varphi \land \psi) \leftrightarrow (\wneg \varphi \lor \wneg\psi)\\[2mm]
(\axnegimp) & \wneg(\varphi \imp \psi) \leftrightarrow(\varphi \land \wneg\psi)\\[2mm]
\end{array} $

\

{\bf Inference rule:}\\

$\begin{array}{ll}
(\MP) & \dfrac{\varphi ~~~\varphi \to\psi}{\psi}\\
\end{array} $
\end{defi}

\noindent
It is worth noting that  axioms (\kax)-(\axouimp), (\axtndcl) and (\axexpl),   together with~(\MP), constitute an adequate Hilbert calculus for classical propositional logic \cpl\ in the signature $\Sigma_c=\{\land,\lor,\imp, \sneg\}$. Moreover, (\kax)-(\axouimp) plus~(\MP) is an adequate Hilbert calculus for classical positive popositional logic \cplp\ in the signature $\Sigma_{cp}=\{\land,\lor,\imp\}$. 

\begin{teo} (\cite[Theorem~4.4.56]{CC16}) \label{adeqMPT0}
The logic $\lptz$  is sound and complete  w.r.t. the matrix logic of \mptz: $\Gamma\vdash_{\lptz} \varphi$ \ iff \ $\Gamma\vDash_{\mptz} \varphi$, for every $\Gamma\cup\{\varphi\} \subseteq {\cal L}_{\Sigma}$.
\end{teo}

\

\noindent
The latter result can be extended to twist-structures semantics, as shown in ~\cite{CFG18}. Indeed,  \lptz\ coincides (up to signature) with $\lfium_\cons$,  an \lfi\ defined over the signature $\Sigma_\cons=\{\land,\lor,\imp, \wneg, \cons\}$  such that the consistency operator $\cons$ is defined as \\

\begin{center}
\begin{tabular}{|c||c|} \hline
$\quad$ & $\cons$ \\
 \hline \hline
     {\bf 1}   & {\bf 1}    \\ \hline
     ${\bf \frac12}$ & {\bf 0}  \\ \hline
     {\bf 0}   & {\bf 1}    \\ \hline
\end{tabular}
\end{center}

\ \\

\noindent
In $\lfium_\cons$ the strong negation $\sneg$ is defined as $\sneg\varphi =_{def} \varphi \to \bot_\varphi$ such that $\bot_\varphi =_{def} (\varphi \wedge \neg\varphi) \wedge \cons\varphi$. On the other hand, the consistency operator $\cons$ is defined in \lptz\ as $\cons \varphi=_{def}\sneg(\varphi \land \wneg \varphi)$.
The twist-structures semantics for $\lfium_\cons$ introduced in~\cite[Definition~9.2]{CFG18} can be adapted to \lptz\ as follows:

\begin{defi} \label{twdom}
Let $\mA=\langle A, \wedge, \vee, \to,\sneg,0,1 \rangle$ be a Boolean algebra.\footnote{In this paper the symbol $\sneg$ will be used  for denoting the strong negation of \lptz\ as well as for denoting the classical negation and its semantical interpretation (the Boolean complement in a Boolean algebra). The context will avoid possible confusions} The {\em twist domain} generated by \mA\ is the set $\tA=\{(z_1,z_2) \in A  \times A \ : \ z_1 \vee z_2 = 1\}$. 
\end{defi}

\begin{defi}  \label{defKlfi1}
Let \mA\ be a Boolean algebra. The  {\em twist structure for \lptz\ over \mA} is the algebra  $\tmA=\langle \tA, \tilde{\wedge}, \tilde{\vee}, \tilde{\imp},\tilde{\sneg},\tilde{\wneg} \rangle$ over $\Sigma$  such that the operations are defined as follows, for every $(z_1,z_2),(w_1,w_2) \in \tA$:

\begin{itemize}
 \item[(i)] $(z_1,z_2)\,\tilde{\wedge}\, (w_1,w_2) = (z_1 \wedge w_1,z_2 \vee w_2)$;
 \item[(ii)]  $(z_1,z_2)\,\tilde{\vee}\, (w_1,w_2) = (z_1 \vee w_1,z_2 \wedge w_2)$;
  \item[(iii)]  $(z_1,z_2)\,\tilde{\imp}\, (w_1,w_2) = (z_1 \imp w_1,z_1 \wedge w_2)$;
 \item[(iv)] $\tilde{{\sim}}(z_1,z_2) = ({\sim}z_1,z_1)$;
\item[(v)] $\tilde{\neg} (z_1,z_2) = (z_2,z_1)$.
\end{itemize}
\end{defi}

\

\noindent
By recalling that the consistency operator $\cons$ is defined in \lptz\ as $\cons \varphi=_{def}\sneg(\varphi \land \wneg \varphi)$, it follows that $\tilde{\circ}(z_1,z_2)=(\sneg(z_1 \land z_2),z_1 \land z_2)$.\footnote{This is why in~\cite[Definition~9.2]{CFG18} clause~(v) was replaced by this clause defining $\tilde{\circ}$.}

\begin{defi}  \label{def-semKlfi1}
The logical matrix associated to the twist structure \tmA\ is $\matA=\langle \tmA,D_\mA\rangle$ where $D_\mA = \{(z_1,z_2) \in \tA \ : \  z_1=1\} = \{(1,a) \ : \  a \in A\}$. The consequence relation associated to \matA\ will be denoted by $\vDash_{\tmA}$. Let $\matM_\lptz = \{\matA \ : \ \mA$ is a Boolean algebra$\}$ be the class of twist models for \lptz. The {\em twist-consequence relation} for \lptz\ is the consequence relation $\vDash_{\matM_\lptz}$ associated to $\matM_\lptz$, namely: $\Gamma \vDash_{\matM_\lptz} \varphi$ iff $\Gamma \vDash_{\tmA} \varphi$ for every Boolean algebra \mA.
\end{defi}

\begin{obs}  \label{twistA2}
In~\cite[Theorem~9.6]{CFG18} it was shown that \lptz\ is sound and complete w.r.t. twist structures  semantics, namely: $\Gamma \vdash_{\lptz} \varphi$ iff $\Gamma\vDash_{\matM_\lptz} \varphi$, for every set of formulas $\Gamma \cup\{\varphi\}$. On the other hand, if $\mathbb{A}_2$  is the two-element Boolean algebra with domain $\{0,1\}$ then $\mathcal{T}_{\mathbb{A}_2}$ consists of three elements: $(1,0)$, $(1,1)$ and $(0,1)$. By identifying these elements with {\bf 1}, ${\bf \frac{1}{2}}$ and {\bf 0}, respectively, then $\mathcal{T}_{\mathbb{A}_2}$ coincides with the three-valued algebra \aptz\ underlying the matrix $\matM_{PT0}$ (recall Definition~\ref{MPT0-def}). Moreover, $\mathcal{MT}_{\mathbb{A}_2}$ coincides with $\matM_{PT0}$. Taking into consideration Theorem~\ref{adeqMPT0}, this situation is analogous to the semantical characterization of  \cpl\ w.r.t. Boolean algebras: it is enough to consider the two-element Boolean algebra $\mathbb{A}_2$.
\end{obs}

\section{The logic \qlptz}

A first-order version of \lptz, called \qlptz, was  proposed in~\cite{CFG19} under the equivalent (up to language) form of $\qlfium_\circ$.\footnote{That is, by taking $\circ$ instead of $\sneg$ as a primitive connective.} For convenience, we reproduce here the main features  of \qlptz. 
 
\begin{defi} \label{fosig} Let $Var=\{v_1,v_2,\ldots\}$ be a denumerable set of individual variables. A first-order signature $\Theta$ for \qlptz\ is given as follows: 
\begin{itemize}
  \item[-] a set $\mathcal{C}$ of individual constants; 
  \item[-] for each $n\geq 1$, a set $\mathcal{F}_n$ of function symbols of arity $n$,
  \item[-] for each $n\geq 1$, a nonempty set $\mathcal{P}_n$ of predicate symbols of arity $n$.
\end{itemize}
\end{defi}

\

The sets of terms and formulas generated by a signature $\Theta$ will be denoted by \tert\ and \fort, respectively. The set of closed formulas (or sentences) and the set of closed terms (terms without variables) over $\Theta$ will be denoted by \sent\ and \ctert, respectively. The formula obtained from a given formula $\varphi$ by substituting every free occurrence of a variable $x$ by a term $t$ will be denoted by $\varphi[x/t]$.

\begin{defi} Let $\Theta$ be a first-order signature. The logic \qlptz\ is obtained from  \lptz\ by adding the following axioms and rules:\\[1mm]

{\bf Axioms Schemas:}\\

$\begin{array}{ll}
{\bf (Ax\exists)} & \varphi[x/t]\to\exists x\varphi, \ \mbox{ if $t$ is a term free for $x$ in $\varphi$}\\[3mm]
{\bf (Ax\forall)} & \forall x\varphi\to\varphi[x/t], \ \mbox{ if $t$ is a term free for $x$ in $\varphi$}\\[3mm]
{\bf (Ax\neg\exists)} & \neg\exists x\varphi\leftrightarrow \forall x \neg\varphi\\[3mm]
{\bf (Ax\neg\forall)} & \neg\forall x\varphi\leftrightarrow \exists x  \neg\varphi
\end{array}$\\

\

{\bf Inference rules:}\\

$\begin{array}{ll}
{\bf (\exists\mbox{\bf -In})} & \dfrac{\varphi\to\psi}{\exists x\varphi\to\psi}, \ \mbox{  where $x$ does not occur free in $\psi$}\\[4mm]
{\bf (\forall\mbox{\bf -In})} & \dfrac{\varphi\to\psi}{\varphi\to\forall x\psi}, \ \mbox{ where $x$ does not occur free in $\varphi$}
\end{array}$
\end{defi}\

The consequence relation of \qlptz\ will be denoted by $\vdash_{\qlptz}$.

\section{Twist structures semantics for \qlptz} \label{swapfol}

In~\cite{CFG19} a  semantics of first-order structures based on  twist structures for $\lfium_\circ$ was proposed for $\qlfium_\circ$. That semantics will be briefly recalled here, adapted to  \qlptz. From now on,  only complete Bolean algebras will be considered.

\begin{defi} \label{tstru} let \mA\ be a complete Boolean algebra. Let \matA\ be the logical matrix associated to a twist structure \tmA\  for \lptz, and let $\Theta$ be a first-order signature (see Definition~\ref{fosig}). A  (first-order) {\em structure} over \matA\  and $\Theta$ (or a {\em\qlptz-structure} over $\Theta$) is pair $\mathfrak{A} = \langle U, I_{\mathfrak{A}} \rangle$ such that $U$ is a nonempty set (the domain or universe of the structure) and $I_{\mathfrak{A}}$ is an interpretation function which assigns:\vspace*{-3mm}
\begin{itemize}
\item[-]  an element $I_\mathfrak{A}(c)$ of $U$ to each individual constant $c \in \mathcal{C}$; 
\item[-] a function $I_\mathfrak{A}(f): U^n \to U$ to each function symbol $f$ of arity $n$; 
\item[-] a function  $I_\mathfrak{A}(P): U^n \to \tA$ to each predicate symbol $P$ of arity $n$.
\end{itemize}
\end{defi}

\begin{Not}From now on, we will write $c^\mathfrak{A}$, $f^\mathfrak{A}$ and $P^\mathfrak{A}$ instead of $I_\mathfrak{A}(c)$, $I_\mathfrak{A}(f)$ and $I_\mathfrak{A}(P)$ to denote the interpretation of an individual constant symbol $c$, a function symbol $f$ and a predicate symbol $P$, respectively.
\end{Not}

\begin{defi} Given  a structure $\mathfrak{A}$ over \matA\ and $\Theta$,  an {\em assignment} over  $\mathfrak{A}$ is any function $\mu: Var \to U$. 
\end{defi}

\begin{defi}~\label{term} 
 Given  a structure $\mathfrak{A}$ over \matA\ and $\Theta$, and given an assignment $\mu: Var \to U$ we define recursively, for each term $t$,  an element $\termvalue{t}^\mathfrak{A}_\mu$ in $U$ as follows:
\begin{itemize}
\item[-] $\termvalue{c}^\mathfrak{A}_\mu = c^\mathfrak{A}$ if $c$ is an individual constant;
\item[-] $\termvalue{x}^\mathfrak{A}_\mu = \mu(x)$ if $x$ is a variable;
\item[-] $\termvalue{f(t_1,\ldots,t_n)}^\mathfrak{A}_\mu = f^\mathfrak{A}(\termvalue{t_1}^\mathfrak{A}_\mu,\ldots,\termvalue{t_n}^\mathfrak{A}_\mu)$ if $f$ is a function symbol of arity $n$ and $t_1,\ldots,t_n$ are terms.
\end{itemize}
\end{defi}

\begin{defi} \label{diaglan} 
Let $\mathfrak{A}$ be  a structure over \matA\ and $\Theta$. The {\em diagram language} of $\mathfrak{A}$ is the set of formulas $For(\Theta_U)$, where  $\Theta_U$ is the signature obtained from $\Theta$ by adding, for each element $a \in U$,  a  new individual constant $\bar{a}$ . 
\end{defi}

\begin{defi} \label{extA}
The structure $\widehat{\mathfrak{A}} = \langle U, I_{\widehat{\mathfrak{A}}} \rangle$ over $\Theta_U$  is the structure $\mathfrak{A}$ over $\Theta$ extended by $I_{\widehat{\mathfrak{A}}}(\bar{a})=a$ for every $a \in A$. 
\end{defi}

\noindent
It is worth noting that $s^{\widehat{\mathfrak{A}}} = s^\mathfrak{A}$ whenever $s$ is a symbol (individual constant, function symbol or predicate symbol) of $\Theta$.

\begin{Not} \label{SentA}
The set of sentences or closed formulas (that is, formulas without free variables) of the diagram language $For(\Theta_U)$ is denoted by $Sen(\Theta_U)$, and the set of terms and of closed terms over $\Theta_U$ will be denoted by $Ter(\Theta_U)$ and $CTer(\Theta_U)$, respectively. If  $t$ is a closed term we can write $\termvalue{t}^\mathfrak{A}$ instead of $\termvalue{t}^\mathfrak{A}_\mu$, for any   assignment $\mu$, since it does not depend on $\mu$.
\end{Not}

\begin{Not} \label{projz}
From now on, if $z \in \tA$ then  $(z)_1$ and $(z)_2$ (or simply $z_1$ and $z_2$) will denote the first and second coordinates of $z$, respectively. 
\end{Not}

\begin{defi}[\qlptz\ interpretation maps]~\label{val}
Let \mA\ be a complete Boolean algebra, and let $\mathfrak{A}$ be  a structure over \matA\ and $\Theta$. The {\em interpretation map} for \qlptz\ over $\mathfrak{A} $ and  \matA\ is a function $\termvalue{\cdot}^{\mathfrak{A}}:Sen(\Theta_U) \to \tA$ satisfying the following clauses (using Notation~\ref{projz} in clauses~(iv) and~(v)):

\begin{enumerate}
\item[($i$)] $\termvalue{P(t_1,\ldots,t_n)}^{\mathfrak{A}} = P^{\mathfrak{A}}(\termvalue{t_1}^{\widehat{\mathfrak{A}}},\ldots,\termvalue{t_n}^{\widehat{\mathfrak{A}}}) $, if $P(t_1,\ldots,t_n)$ is atomic; \
\item[($ii$)] $\termvalue{\#\varphi}^{\mathfrak{A}} = \tilde{\#} \termvalue{\varphi}^{\mathfrak{A}}$, for every $\#\in \{\neg, \sneg\}$;
\item[($iii$)]  $\termvalue{\varphi \# \psi}^{\mathfrak{A}} = \termvalue{\varphi}^{\mathfrak{A}} \, \tilde{\#} \, \termvalue{\psi}^{\mathfrak{A}}$, for every $\#\in \{\wedge,\vee, \to\}$;
\item[($iv$)] $\termvalue{\forall x\varphi}^{\mathfrak{A}} = \big(\bigwedge_{a \in U} (\termvalue{\varphi[x/\bar{a}]}^{\mathfrak{A}})_1,\bigvee_{a \in U} (\termvalue{\varphi[x/\bar{a}]}^{\mathfrak{A}})_2\big)$.
\item[($v$)] $\termvalue{\exists x\varphi}^{\mathfrak{A}} = \big(\bigvee_{a \in U} (\termvalue{\varphi[x/\bar{a}]}^{\mathfrak{A}})_1,\bigwedge_{a \in U} (\termvalue{\varphi[x/\bar{a}]}^{\mathfrak{A}})_2\big)$.
\end{enumerate}
\end{defi}

\begin{obs} \label{TA-lat}
A partial order can be naturally introduced in $\tmA$ as follows:  $z \leq w$ iff $z_1 \leq w_1$ and $z_2 \geq w_2$. It is easy to see that, with this order, \tmA\ is a complete lattice (since $\mA$ is a complete Boolean algebra), in which
\begin{itemize}
\item[] $\bigwedge_{i \in I} z_i = \big(\bigwedge_{i \in I} (z_i)_1,\bigvee_{i \in I} (z_i)_2\big)$, and
\item[] $\bigvee_{i \in I} z_i = \big(\bigvee_{i \in I} (z_i)_1,\bigwedge_{i \in I} (z_i)_2\big)$. 
\end{itemize}
Note that ${\bf 1}=_{def}(1,0)$ and ${\bf 0}=_{def}(0,1)$ are the top and bottom elements of \tmA, respectively. These considerations justify the definition of the interpretation of the quantifiers given in Definition~\ref{val}(iv) and~(v).
\end{obs}

\

\noindent Recall the notation stated in Definition~\ref{diaglan}. The interpretation map can be extended to arbitrary formulas as follows:

\begin{defi} \label{vmu} Let \mA\ be a complete Boolean algebra, and let  $\mathfrak{A}$ be a structure over \matA\ and $\Theta$. Given an assignment $\mu$ over $\mathfrak{A} $, the {\em extended interpretation map} $\termvalue{\cdot}^{\mathfrak{A}}_{\mu}: For(\Theta_U) \to \tA$ is given by $\termvalue{\varphi}^{\mathfrak{A}}_{\mu} = \termvalue{\varphi[x_1/\overline{\mu(x_1)}, \ldots, x_n/\overline{\mu(x_n)}]}^{\mathfrak{A}}$, provided that the free variables of $\varphi$ occur in $\{x_1, \ldots, x_n \}$.
\end{defi}

\begin{defi} \label{consrel0} Let \mA\ be a complete Boolean algebra, and let $\mathfrak{A}$ be  a structure over \matA\ and $\Theta$.  Given a set of formulas $\Gamma \cup \{\varphi\} \subseteq For(\Theta_U)$, $\varphi$ is said to be  a {\em semantical consequence of $\Gamma$ w.r.t. $(\mathfrak{A}, \matA)$}, denoted by $\Gamma\models_{(\mathfrak{A}, \matA)}\varphi$, if the following holds:  if $\termvalue{\gamma}^{\mathfrak{A}}_{\mu} \in D$,  for every formula $\gamma \in \Gamma$ and every assignment $\mu$, then $\termvalue{\varphi}^{\mathfrak{A}}_{\mu} \in D$,  for every assignment $\mu$.
\end{defi}

\begin{defi} [Semantical consequence relation in \qlptz\ w.r.t. twist structures] \label{consrel} Let $\Gamma \cup \{\varphi\} \subseteq For(\Theta)$ be a set of formulas. Then  $\varphi$ is said to be  a {\em semantical consequence of $\Gamma$ in \qlptz\ w.r.t. first-order twist structures}, denoted by $\Gamma\models_{\qlptz}\varphi$, if $\Gamma\models_{(\mathfrak{A}, \matA)}\varphi$ for every pair $(\mathfrak{A}, \matA)$.
\end{defi}

\begin{teo} [Adequacy of \qlptz\ w.r.t. first-order twist structures (\cite{CFG19})] \label{adeq-Qlptz-twist}
For every set $\Gamma \cup\{ \varphi\}  \subseteq For(\Theta)$:   $\Gamma \vdash_\qlptz \varphi$ if and only if  $\Gamma \models_\qlptz \varphi$.\footnote{As observed above, in~\cite{CFG19} the logic $\qlfium_\circ$ was analyzed instead of \qlptz. However, both logics are equivalent, the only difference being the use of  $\circ$ instead of $\sneg$ as primitive connective. The adaptation of the adequacy result for  $\qlfium_\circ$ given in~\cite{CFG19} to the logic \qlptz\ is straightforward.}
\end{teo}

\ 

\noindent
In Remark~\ref{twistA2} was observed that $\mathcal{T}_{\mathbb{A}_2}$, the twist structure for \lptz\ defined over the two-element Boolean algebra $\mathbb{A}_2$, coincides (up to names) with the three-valued algebra \aptz\ underlying the matrix $\matM_{PT0}$ and, moreover, $\mathcal{MT}_{\mathbb{A}_2}$ coincides with the three-valued characteristic matrix $\matM_{PT0}$ of \lptz. In~\cite{CFG19}  it was  proven that \qlptz\ can be characterized by first-order structures defined  over $\matM_{PT0}$.\footnote{Once again, it is worth observing that the result obtained in~\cite{CFG19} concerns the logic $\qlfium_\circ$ instead of \qlptz.}

\begin{teo} [Adequacy of \qlptz\ w.r.t. first-order structures over $\matM_{PT0}$ (\cite{CFG19})] \label{adeq-qmbcM5}
For every set $\Gamma \cup \{\varphi\} \subseteq For(\Theta)$:  $\Gamma\vdash_{\qlptz} \varphi$ \ iff \ $\Gamma\models_{(\mathfrak{A}, \matM_{PT0})} \varphi$ for every  structure $\mathfrak{A}$ over $\Theta$ and $\matM_{PT0}$.
\end{teo}

\begin{obs}
It is worth observing that Theorem~\ref{adeq-qmbcM5} constitutes a variant of the  adequacy theorem of first-order \dacdot\ w.r.t. first-order structures given in~\cite{dot:85a}. Indeed, both logics are the same  (up to language), and the semantic structures are the same, up to presentation.
\end{obs}

\section{Twist-valued models for set theory} \label{twistm}

As mentioned before, a  three-valued model for a paraconsistent set theory based on lattice-valued models for \ZF, as a non-classical variant of the  well-known Scott-Solovay-Vop$\check{e}$nka Boolean-valued models for \ZF, was proposed by L\"owe and Tarafder in \cite{low:tar:15}. Specifically, they introduce a  three-valued logic called \pst\ which can be expanded with a paraconsistent negation $\neg$ (which they denote by $\ast$) and then a model for \ZF\ is constructed over the three-valued algebra \pst, as well as over its expansion $(\pst,\neg)$, along the same lines as the traditional Boolean-valued models. It is known that the logic $(\pst,\neg)$, introduced in~\cite{ConSil:14} as \mpt, coincides up to language with \lptz. We will return to this point in Section~\ref{genPS3}.

In this section, a twist-valued model for a paraconsistent set theory \ZFlp\  based on \qlptz\ will be defined, for any complete Boolean algebra \mA. It will be shown that this models constitute a generalization of the Boolean-valued models for set theory, as well as of  L\"owe-Tarafder's three-valued model. Our constructions, as well as the proof of their formal properties, are entirely based on the exposition of  Boolean-valued models given in the book~\cite{bell}, which constitutes a fundamental reference to this subject.

Consider the first order signature $\Theta_{\ZF}$ for set theory \ZF\ which consists of two binary predicates $\epsilon$ (for membership) and $\approx$ (for identity).
The logic \ZFlp\ will be defined over the first-order language $\mathcal{L}$ generated by  $\Theta_{\ZF}$ based on the signature of \qlptz, that is: the set of connectives is $\Sigma=\{\land,\lor,\imp, \sneg, \wneg\}$, together with the quantifiers $\forall$ and $\exists$ and the set $Var=\{v_1,v_2,\ldots\}$ of individual variables. As usual, $dom(f)$ and $ran(f)$ will thenote the {\em domain} and {\em image} (or {\em rank}) of a given function $f$. 
 
\begin{defi} \label{vTAdef}
Let \mA\ be a complete Boolean algebra, and let $\alpha$ be an ordinal number. Define, by transfinite recursion on $\alpha$, the following:

\begin{itemize}
\item[] ${\bf V}^{\tmA}_\alpha=\{x \ : \ x \mbox{ is a function and } ran(x) \subseteq \tA \mbox{ and }  dom(x) \subseteq {\bf V}^{\tmA}_\xi \mbox{ for some } \xi  < \alpha\}$;
\item[] ${\bf V}^{\tmA}=\{x \ : \ x \in {\bf V}^{\tmA}_\alpha  \mbox{ for some } \alpha\}$.
\end{itemize}
The class ${\bf V}^{\tmA}$ is called the {\em twist-valued model} over the complete Boolean algebra \mA.
\end{defi}

\begin{defi} \label{def-langs} 
Expand the language $\mathcal{L}$ by adding a  constant $\bar{u}$ to each element $u$ of  ${\bf V}^{\tmA}$, obtaining a language denoted by $\mathcal{L}(\tmA)$. The fragments of $\mathcal{L}$ and $\mathcal{L}(\tmA)$ without the connective $\neg$ will be denoted by  $\mathcal{L}_p$ and $\mathcal{L}_p(\tmA)$, respectively. They will be called the {\em pure \ZF-languages}. Observe that $\mathcal{L}(\tmA)$ and $\mathcal{L}_p(\tmA)$ are proper classes. Finally, a formula $\varphi$ in $\mathcal{L}_p$ is called {\em restricted} if every occurrence of a quantifier in $\varphi$ is of the form $\forall x(x \in y \to \ldots)$ or $\exists x(x \in y \land \ldots)$, or if it is proved to be equivalent in \ZFC\ to a formula of this kind. 
\end{defi}

\begin{Not} \label{consVTA}
By simplicity, and as it is done with Boolean-valued models, we will identify the element $u$ of  ${\bf V}^{\tmA}$ with its name  $\bar{u}$ in $\mathcal{L}(\tmA)$, simply writting $u$. Moreover, if $\varphi$ is a formula in which $x$ is the unique variable (possibly) occurring free, we will write $\varphi(u)$ instead of $\varphi[x/u]$ or $\varphi[x/\bar{u}]$.
\end{Not}

\begin{obs} [Induction principles] \label{IP}
Recall that, from the regularity axiom of \ZF, the sets ${\bf V}_\alpha=\{ x \ : \ x \subseteq {\bf V}_\xi  \mbox{ for some } \xi  < \alpha\}$ are definable for every ordinal $\alpha$. Moreover, in \ZF\ every set $x$ belongs to some ${\bf V}_\alpha$. This induces a function $rank(x)=_{def}$ least $\alpha$ such that $x \in {\bf V}_\alpha$. Since $rank(x) < rank(y)$ is well-founded, it induces a {\em principle of induction on rank}:
\begin{quote}
Let  $\Psi$ be a property over sets. Assume, for every set $x$, the following:
if $\Psi(y)$ holds for every $y$ such that $rank(y) <rank(x)$ then $\Psi(x)$ holds. Hence, $\Psi(x)$ holds for every $x$.
\end{quote}
From this, the following {\em Induction Principle} (IP) holds in ${\bf V}^{\tmA}$ (similar to the one for Boolean-valued models): 
\begin{quote}
Let  $\Psi$ be a property over individuals in  ${\bf V}^{\tmA}$. Assume, for every $x \in {\bf V}^{\tmA}$, the following:
if $\Psi(y)$ holds for every $y \in dom(x)$ then $\Psi(x)$ holds. Hence, $\Psi(x)$ holds for every $x \in  {\bf V}^{\tmA}$.
\end{quote}
Both induction principles are fundamental tools in order to prove properties in ${\bf V}^{\tmA}$. 
\end{obs}

\begin{defi} \label{defint-twist}
Define  by induction on the complexity in  $\mathcal{L}(\tmA)$ a mapping $\termvalue{\cdot}^{{\bf V}^{\tmA}}$ (or simply $\termvalue{\cdot}$) assigning to each closed formula  in $\mathcal{L}(\tmA)$ a value in $\tA$ as follows:

$$
\begin{array}{lll}
\termvalue{u \pert v}&=& \displaystyle \bigvee_{x \in dom(v)} (v(x) \, \tilde{\land} \, \termvalue{x\approx u})\\[6mm]
&=& \displaystyle \big(\bigvee_{x \in dom(v)} ((v(x))_1 \land \termvalue{x\approx u}_1), \bigwedge_{x \in dom(v)} ((v(x))_2 \lor \termvalue{x\approx u}_2)\big)\\[6mm]
\termvalue{u \approx v}&=& \displaystyle \bigwedge_{x \in dom(u)} (u(x) \, \tilde{\to} \, \termvalue{x\pert v}) \ \ \tilde{\land} \ \ \bigwedge_{x \in dom(v)} (v(x) \, \tilde{\to} \, \termvalue{x\pert u})\\[6mm]
&=& \displaystyle \big(\bigwedge_{x \in dom(u)} ((u(x))_1 \to \termvalue{x\pert v}_1), \bigvee_{x \in dom(u)} ((u(x))_1 \land \termvalue{x\pert v}_2)\big)\\[6mm]
&&\tilde{\land} \ \  \displaystyle  \big(\bigwedge_{x \in dom(v)} ((v(x))_1 \to \termvalue{x\pert u}_1), \bigvee_{x \in dom(v)} ((v(x))_1 \land \termvalue{x\pert u}_2)\big)\\[6mm]
\termvalue{\phi \# \psi} &=&  \termvalue{\phi} \tilde{\#} \termvalue{\psi} \ \ \ \mbox{ for $\# \in \{\land, \lor,\to\}$}\\[6mm]
\termvalue{\# \psi} &=&  \tilde{\#} \termvalue{\psi} \ \ \ \mbox{ for $\# \in \{\sneg, \wneg\}$}\\[6mm]
\termvalue{\forall x \varphi(x)} &=& \displaystyle \bigwedge_{u \in {\bf V}^{\tmA}} \termvalue{\varphi(u)} = \big(\bigwedge_{u \in {\bf V}^{\tmA}} \termvalue{\varphi(u)}_1,\bigvee_{u \in {\bf V}^{\tmA}} \termvalue{\varphi(u)}_2\big) \\[6mm]
\termvalue{\exists x \varphi(x)} &=& \displaystyle \bigvee_{u \in {\bf V}^{\tmA}} \termvalue{\varphi(u)} = \big(\bigvee_{u \in {\bf V}^{\tmA}} \termvalue{\varphi(u)}_1,\bigwedge_{u \in {\bf V}^{\tmA}} \termvalue{\varphi(u)}_2\big).
\end{array}
$$
$\termvalue{\varphi}^{{\bf V}^{\tmA}}$ is called the {\em twist truth-value} of the  sentence $\varphi \in \mathcal{L}(\tmA)$ in the twist-valued model ${\bf V}^{\tmA}$ over the complete Boolean algebra \mA.
\end{defi}

\begin{obs}
Observe that ${\bf V}^{\tmA}$ can be seen as a structure for \qlptz\ over  \matA\ and $\Theta_\ZF$ in a  wide sense, given that its domain is  a proper class. Under this identification, the twist truth-value $\termvalue{\varphi}^{{\bf V}^{\tmA}}$ of the sentence $\varphi$ in ${\bf V}^{\tmA}$ is exactly the value assigned to $\varphi$ by the interpretation map for \qlptz\ over  ${\bf V}^{\tmA}$ and  \matA\ (recall Definition~\ref{val}). In this case we assume that the mappings $(\cdot\pert\cdot)^{{\bf V}^{\tmA}}$ and $(\cdot\approx\cdot)^{{\bf V}^{\tmA}}$ are as in Definition~\ref{defint-twist}. 
\end{obs}



\

\noindent Recall the notion of semantical consequence relation in \qlptz\ (see Definitions~\ref{consrel0} and~\ref{consrel}). This motivates the following:

\begin{defi} \label{valVTA} A sentence $\varphi$  in $\mathcal{L}(\tmA)$ is said to be {\em valid in ${\bf V}^{\tmA}$}, which is denoted by
${\bf V}^{\tmA} \models \varphi$,  if $\termvalue{\varphi}^{{\bf V}^{\tmA}} \in D_{\mA}$.
\end{defi}

\

\noindent
The semantical notions introduced above can  easily be generalized to formulas with free variables. Recall from Notation~\ref{consVTA} that $\overline{u}$ is identified with $u$ in  ${\bf V}^{\tmA}$. Then:

\begin{defi} Let $\varphi$ be a formula in  $\mathcal{L}$ whose  free variables occur in $\{x_1, \ldots, x_n \}$. Given a twist-valued model  ${\bf V}^{\tmA}$ and an assignment $\mu:Var \to {\bf V}^{\tmA}$, the  {\em twist truth-value} of  $\varphi$ in  ${\bf V}^{\tmA}$ and $\mu$ is defined as follows: $\termvalue{\varphi}^{{\bf V}^{\tmA}}_\mu=_{def}
\termvalue{\varphi[x_1/\mu(x_1), \ldots, x_n/\mu(x_n)]}^{{\bf V}^{\tmA}}$. The formula  $\varphi$ is {\em valid} in  ${\bf V}^{\tmA}$  if $\termvalue{\varphi}^{{\bf V}^{\tmA}}_\mu \in  D_{\mA}$ for every $\mu$.
\end{defi}

\begin{defi} \ZFlp\ is the logic of the class of twist-valued models, seen as \qlptz-structures over the signature $\Theta_{\ZF}$. That is, \ZFlp\ is the set of formulas of $\mathcal{L}$ which are valid in every twist-valued model ${\bf V}^{\tmA}$. 
\end{defi}

\section{Boolean-valued models versus twist-valued models} \label{twist-ZFC}

In this section, the relationship between twist-valued models and Boolean-valued models will be briefly analized. It will be shown that these models enjoy similar properties than the Boolean-valued models (when restricted to pure \ZF-languages). These similarities  will be fundamental in order to prove that  \ZFC\ is valid w.r.t. twist-valued models (see Theorem~\ref{modelZFC} below).

The following basic results for twist-valued models are analogous to the corresponding ones for Boolean-valued models obtained in~\cite[Theorem~1.17]{bell}. All these results will be proven by using the Induction Principle (IP) (recall Remark~\ref{IP}). From now on we assume that the reader is familiar with the book~\cite{bell}.

\begin{lema} \label{lemregu}
Let \mA\ be a complete Boolean algebra, and let $u\in {\bf V}^{\tmA}$. Then  $\termvalue{u \in u}_1=0$.
\end{lema}
\begin{proof}
Assume the inductive hypothesis $\termvalue{y \in y}_1=0$ for every $y \in dom(u)$. Note that
$$
\termvalue{u \pert u}_1=  \bigvee_{y \in dom(u)} ((u(y))_1 \land \termvalue{y\approx u}_1).
$$
Let  $y \in dom(u)$. Then
$$
\begin{array}{lll}
(u(y))_1 \land \termvalue{y \approx u}_1 &\leq& \displaystyle (u(y))_1 \land  \bigwedge_{x \in dom(u)} ((u(x))_1 \to \termvalue{x\pert y}_1)\\[2mm]
&\leq& (u(y))_1 \land  ((u(y))_1 \to \termvalue{y\pert y}_1)\\[2mm] 
&\leq&   \termvalue{y\pert y}_1 = 0.
\end{array}
$$
Then $u(y)_1 \land \termvalue{y \approx u}_1=0$ for every $y \in dom(u)$, hence  $\termvalue{u \in u}_1=0$.
\end{proof}

\begin{teo} \label{basicVB}
Let \mA\ be a complete Boolean algebra, and let $u,v,w \in {\bf V}^{\tmA}$. Then:\\[1mm]
(i) $\termvalue{u \approx u}_1=1$.\\[1mm]
(ii) $u(x)_1 \leq \termvalue{x \pert u}_1$, for every $x \in dom(u)$.\\[1mm]
(iii) $\termvalue{u \approx v}_1 = \termvalue{v \approx u}_1$.\\[1mm]
(iv) $\termvalue{u \approx v}_1 \land \termvalue{v \approx w}_1 \leq \termvalue{u \approx w}_1$.\\[1mm]
(v) $\termvalue{u \approx v}_1 \land \termvalue{u \pert w}_1 \leq \termvalue{v \pert w}_1$.\\[1mm]
(vi) $\termvalue{v \approx w}_1 \land \termvalue{u \pert v}_1 \leq \termvalue{u \pert w}_1$.\\[1mm]
(vii) $\termvalue{u \approx v}_1 \land \termvalue{\varphi(u)}_1 \leq \termvalue{\varphi(v)}_1$ for every formula $\varphi(x)$ in  $\mathcal{L}_p(\tmA)$.
\end{teo}
\begin{proof}
The proof of items~(i)-(vi) is analogous to the proof of the corresponding items found in~\cite[Theorem~1.17]{bell}. The proof of item~(vii) is easily done by induction on the complexity of $\varphi(x)$ by observing that: the proof when $\varphi$ is atomic uses Lemma~\ref{lemregu}, for $\varphi=(x \pert x)$, and items (i)-(vi) for the other cases. For complex formulas the result follows easily by induction hypothesis.
\end{proof}

\begin{lema} \label{lemaBQ}
Let \mA\ be a complete Boolean algebra. Then, for every formula $\varphi(x)$ in  $\mathcal{L}_p(\tmA)$ and every $u\in {\bf V}^{\tmA}$: $\termvalue{\exists y((u \approx y) \land \varphi(y))}_1 =\termvalue{\varphi(u)}_1$.
\end{lema}
\begin{proof}
It follows from Theorem~\ref{basicVB} items~(i), (iii) and~(viii). Indeed,
$$
\begin{array}{lll}
\termvalue{\exists y((u \approx y) \land \varphi(y))}_1 &=& \displaystyle \bigvee_{x \in dom(u)} (\termvalue{u\approx y}_1 \land \termvalue{\varphi(y)}_1)\\[3mm]
&\leq& \termvalue{\varphi(u)}_1 = \termvalue{u\approx u}_1 \land \termvalue{\varphi(u)}_1\\[3mm] 
&\leq&   \termvalue{\exists y((u \approx y) \land \varphi(y))}_1.
\end{array}
$$
\end{proof}

\begin{Not}
The following notation from~\cite{bell} will be adopted from now on:
$$\exists x \pert u \, \varphi(x) =_{def} \exists x(x \pert u \land \varphi(x));$$
$$\forall x \pert u \, \varphi(x) =_{def} \forall x(x \pert u \to \varphi(x)).$$
\end{Not}

\begin{teo} \label{BQ}
Let \mA\ be a complete Boolean algebra. Then, for every formula $\varphi(x)$ in  $\mathcal{L}_p(\tmA)$ and every $u\in {\bf V}^{\tmA}$:
$$\termvalue{\exists x \pert u \, \varphi(x)}_1 = \bigvee_{x \in dom(u)} ((u(x))_1 \wedge \termvalue{\varphi(x)}_1)$$
and
$$\termvalue{\forall x \pert u \, \varphi(x)}_1 = \bigwedge_{x \in dom(u)} ((u(x))_1 \to \termvalue{\varphi(x)}_1).$$
\end{teo}
\begin{proof}
The proof is similar to that for~\cite[Corollary~1.18]{bell}, taking into account Theorem~\ref{basicVB} and Lemma~\ref{lemaBQ}
\end{proof}

\

\noindent
Recall that a complete Boolean algebra \mA' is a complete subalgebra of the  complete Boolean algebra \mA\ provided that \mA' is a subalgebra of  \mA\ and $\bigvee_{\mA'} X=\bigvee_\mA X$ and  $\bigwedge_{\mA'} X=\bigwedge_\mA X$ for every $X \subseteq |\mA'|$. Analogously, we say that a twist-structure \tmpA\ is a {\em complete subalgebra} of the twist-structure \tmA\ if \tmpA\ is a subalgebra of  \tmA\ and $\bigvee_{\tmpA} X = \bigvee_{\tmA} X$ and  $\bigwedge_{\tmpA} X = \bigwedge_{\tmA} X$ for every $X \subseteq |\tmpA|$, recalling Remark~\ref{TA-lat}. 

\begin{prop}
If \mA' is a complete subalgebra of \mA\ then \tmpA\ is a complete subalgebra of \tmA.
\end{prop} 
\begin{proof}
If follows from Definition~\ref{defKlfi1}  and Remark~\ref{TA-lat}.
\end{proof}

\begin{teo} \label{th1.20}
Let \mA' be a complete subalgebra of the  complete Boolean algebra \mA. Then:\\[1mm]
(i)  ${\bf V}^{\tmpA} \subseteq {\bf V}^{\tmA}$.\\[1mm]
(ii) for every $u,v \in {\bf V}^{\tmpA}$:  $\termvalue{u \pert w}^{{\bf V}^{\tmpA}} = \termvalue{u \pert w}^{{\bf V}^{\tmA}}$, and $\termvalue{u \approx w}^{{\bf V}^{\tmpA}} = \termvalue{u \approx w}^{{\bf V}^{\tmA}}$.
\end{teo}

\begin{coro} \label{valequal}
Suppose that \mA' is a complete subalgebra of \mA. Then, for any restricted formula $\varphi(x_1,\ldots,x_n)$ in $\mathcal{L}_p$ (recall Definition~\ref{def-langs}) and for every $u_1,\ldots,u_n \in \tmpA$: $\termvalue{\varphi(u_1,\ldots,u_n)}^{{\bf V}^{\tmpA}} = \termvalue{\varphi(u_1,\ldots,u_n)}^{{\bf V}^{\tmA}}$.
\end{coro}
\begin{proof}
The proof is analogous to that for~\cite[Corollary~1.21]{bell}.
\end{proof}

\begin{obs} \label{remVA2}
Recall from Remark~\ref{twistA2} that $\mathcal{T}_{\mathbb{A}_2}$, the twist structure for \lptz\ defined over the two-element Boolean algebra $\mathbb{A}_2$, coincides (up to names) with the three-valued algebra \aptz\ underlying the matrix $\matM_{PT0}$, where {\bf 1}, ${\bf \frac{1}{2}}$ and {\bf 0} are identified with $(1,0)$, $(1,1)$ and $(0,1)$, respectively. Hence, the twist-valued structure ${\bf V}^{\mathcal{T}_{\mathbb{A}_2}}$ will be denoted by  ${\bf V}^{\aptz}$  Since  $\mathbb{A}_2$ is a complete subalgebra of any complete Boolean algebra \mA\ then ${\bf V}^{\aptz}$ is  a complete subalgebra of ${\bf V}^{\tmA}$, for any \tmA. By Theorem~\ref{th1.20}, $\termvalue{u \pert v}^{{\bf V}^{\aptz}} = \termvalue{u \pert v}^{{\bf V}^{\tmA}}$ and $\termvalue{u \approx v}^{{\bf V}^{\aptz}} = \termvalue{u \approx v}^{{\bf V}^{\tmA}}$ for every $u,v \in {\bf V}^{\aptz}$ and every \tmA. As happens with the Boolean-valued model ${\bf V}^{\mathbb{A}_2}$, the twist-valued model ${\bf V}^{\aptz}$ is, in some sense, isomorphic to the standard universe ${\bf V}$, as it will be shown in Theorem~\ref{th1.23} below.
\end{obs}

\begin{defi}
Define by transfinite recursion on the well-founded relation $y \in x$ the following, for each $x \in {\bf V}$: $\hat{x} =_{def} \{\langle  \hat{y},{\bf 1}\rangle \ : \ y \in x\}$.
\end{defi}

\

\noindent
It is clear that $\hat{x} \in {\bf V}^{\aptz}$ and so $\hat{x} \in {\bf V}^{\tmA}$ for every \tmA.  Hence,  if $\varphi(v_1,\ldots,v_n)$ is a restricted formula in $\mathcal{L}_p$ and $x_1,\ldots,x_n \in {\bf V}$ then $\termvalue{\varphi(\widehat{x_1},\ldots,\widehat{x_n})}^{{\bf V}^{\aptz}}=\termvalue{\varphi(\widehat{x_1},\ldots,\widehat{x_n})}^{{\bf V}^{\tmA}}$ for every \tmA, by Corollary~\ref{valequal}.

\begin{lema} \label{val01}
Let $\varphi(v_1,\ldots,v_n)$ be a formula in $\mathcal{L}_p$, and let $x_1,\ldots,x_n \in {\bf V}$. Then, $\termvalue{\varphi(\widehat{x_1},\ldots,\widehat{x_n})}^{{\bf V}^{\aptz}} \in \{{\bf 0},{\bf 1}\}$.
\end{lema}
\begin{proof}
The result is proven by induction on the complexity of $\varphi$.
\end{proof}

\begin{coro} \label{corval01}
Let $\varphi(v_1,\ldots,v_n)$ be a restricted formula in $\mathcal{L}_p$, and let $x_1,\ldots,x_n \in {\bf V}$. Then, $\termvalue{\varphi(\widehat{x_1},\ldots,\widehat{x_n})}^{{\bf V}^{\tmA}} \in \{{\bf 0},{\bf 1}\}$ for every \mA.
\end{coro}
\begin{proof}
It follows by Lemma~\ref{val01} and by Corollary~\ref{valequal}.
\end{proof}

\begin{teo} \ \\[1mm] \label{th1.23}
(i) For every $x \in {\bf V}$ and $u \in {\bf V}^{\tmA}$: $\displaystyle \termvalue{u \pert \hat{x}} = \bigvee_{y \in x} \termvalue{u \approx \hat{y}}$.\\[2mm]
(ii) For $x,y \in {\bf V}$: 
\begin{itemize}
\item[] $x \in y$ holds in \ZFC\ iff \ ${\bf V}^{\tmA} \models (\hat{x} \pert \hat{y})$ for every \mA;
\item[] $x = y$ holds in \ZFC\ iff \ ${\bf V}^{\tmA} \models (\hat{x} \approx \hat{y})$ for every \mA.
\end{itemize}
(iii) The function $x \mapsto \hat{x}$ is one-to-one from ${\bf V}$ to ${\bf V}^{\aptz}$.\\[2mm]
(iv) For every $u \in {\bf V}^{\aptz}$ there is a (unique) $x \in {\bf V}$ such that ${\bf V}^{\tmA} \models (u \approx \hat{x})$ for all \mA.\\[2mm]
(v) For every formula $\varphi(v_1,\ldots,v_n)$ in $\mathcal{L}_p$ and every $x_1,\ldots,x_n \in {\bf V}$: 
\begin{itemize}
\item[] $\varphi(x_1,\ldots,x_n)$ holds in \ZFC\ iff \ ${\bf V}^{\aptz} \models \varphi(\widehat{x_1},\ldots,\widehat{x_n})$.
\end{itemize}
In addition if $\varphi$ is restricted (recall Definition~\ref{def-langs}) then, for every $x_1,\ldots,x_n \in {\bf V}$:
\begin{itemize}
\item[]  $\varphi(x_1,\ldots,x_n)$ holds in \ZFC\ iff \ ${\bf V}^{\tmA} \models \varphi(\widehat{x_1},\ldots,\widehat{x_n})$, for every \mA.
\end{itemize} 
\end{teo}
\begin{proof}
It follows by an easy adaptation of the proof of~\cite[Theorem~1.23]{bell}. The only points to be considered are the following:\\[1mm]
(i) Note that ${\bf 1} \, \tilde{\land} \, a = a$ for every $a \in |\tmA|$. Then, the adaptation of the proof of this item is immediate.\\[1mm]
(ii) Both assertions are simultaneously proven by induction on $rank(y)$ (see Remark~\ref{IP}), where the induction hypothesis is: for every $z$ with $rank(z) < rank(y)$, $x \in z$ iff  ${\bf V}^{\tmA} \models (\hat{x} \pert \hat{z})$ for every $x$ and \mA; $x = z$ iff  ${\bf V}^{\tmA} \models (\hat{x} \approx \hat{z})$ for every $x$ and \mA; and  $z \in x$ iff  ${\bf V}^{\tmA} \models (\hat{z} \pert \hat{x})$ for every $x$ and \mA.  For the first assertion,  Corollary~\ref{corval01} should  be used. For the second assertion, note that $1\to a = a$ for every $a \in |\mA|$. Hence $\displaystyle (\termvalue{\hat{x} \approx \hat{z}}^{{\bf V}^{\tmA}})_1 = \bigwedge_{y \in x} (\termvalue{\hat{y} \pert \hat{z}}^{{\bf V}^{\tmA}})_1 \, \land \, \bigwedge_{y \in z} (\termvalue{\hat{y} \pert \hat{x}}^{{\bf V}^{\tmA}})_1$. Use then the first assertion, induction hypothesis and the axiom of extensionality. \\[1mm]
(iii) It follows from~(ii). \\[1mm] 
(iv) By adapting the proof of~\cite[Theorem~1.23(iv)]{bell}, at some point of the proof the set  $v=\{ y \in {\bf V} \ : \ u(x)={\bf 1} \mbox{ and } (\termvalue{x \approx \hat{y}}^{{\bf V}^{\tmA}})_1={\bf 1}$, for some $x \in dom(u)\}$  of ${\bf V}$ must be considered.\\[1mm]
(v) In order to  adapt the proof of~\cite[Theorem~1.23(v)]{bell} it should be noted that, if $\emptyset \neq X \subseteq |\aptz|$ is such that $\bigvee_{\aptz} X=1$, then $1 \in X$. 
From this, the inductive step $\varphi=\exists x \psi$ can be treated analogously to the proof of~\cite[Theorem~1.23(v)]{bell}. In addition, the use of the Leibniz rule (see~\cite[Theorem~1.17(vii)]{bell}) at this point of the proof  can be adapted here to an application of Theorem~\ref{basicVB}(vii) as follows:
$$1=(\termvalue{\psi(x,\widehat{x_1},\ldots,\widehat{x_n}}^{{\bf V}^{\aptz}})_1 \land (\termvalue{x \approx \hat{y}}^{{\bf V}^{\aptz}})_1 \leq (\termvalue{\psi(\hat{y},\widehat{x_1},\ldots,\widehat{x_n}}^{{\bf V}^{\aptz}})_1.$$
Hence $(\termvalue{\psi(\hat{y},\widehat{x_1},\ldots,\widehat{x_n}}^{{\bf V}^{\aptz}})_1=1$, and the rest of the proof follows from here.
\end{proof}

\

\noindent
Now it will be shown the {\em Maximum Principle} of Boolean-valued models (see~\cite[Lemma~1.27]{bell}) is also valid in twist-valued models. The adaptation to our framework of the proof of this result found in~\cite{bell} is straightfoward.

\begin{defi} \label{defmix}
Let \mA\ be a complete Boolean algebra. Given sets $E=\{a_i \ : \ i \in I\} \subseteq |\mA|$ and  $F=\{u_i \ : \ i \in I\} \subseteq {\bf V}^{\tmA}$, the {\em twist mixture} of $F$ with respect to $E$ is the element $u=\sum_{i \in I} a_i \odot u_i$ of ${\bf V}^{\tmA}$ defined as follows:\footnote{It is worth observing that the definition of the second coordinate of $u(z)$ will be irrelevant.}
\begin{itemize}
\item[] $dom(u) =  \displaystyle \bigcup_{i \in I} dom(u_i)$, and
\item[] $u(z)= \displaystyle \big( \bigvee_{i \in I}(a_i \land \termvalue{z \pert u_i}_1), \sneg\bigvee_{i \in I}(a_i \land \termvalue{z \pert u_i}_1) \big)$, for every $z \in dom(u)$.
\end{itemize}
\end{defi}

\begin{lema} [Mixing Lemma] \label{mix-lem}
Let $\{a_i \ : \ i \in I\} \subseteq |\mA|$ and  $\{u_i \ : \ i \in I\} \subseteq {\bf V}^{\tmA}$, and let $u=\sum_{i \in I} a_i \odot u_i$. Suppose that, for every $i,j \in I$, $a_i \land a_j \leq \termvalue{u_i \approx u_j}_1$.  Then $a_i \leq \termvalue{u \approx u_i}_1$ for every $i \in I$. 
\end{lema}
\begin{proof}
It can be proved by a straightforward adaptation of the proof of~\cite[Lemma~1.25]{bell}, taking into account Theorem~\ref{basicVB} items~(ii), (iii) and~(vi).
\end{proof}

\

\noindent The next fundamental result shows that the set of pure \ZF-sentences validated by each twist-valued structure ${\bf V}^{\tmA}$ is a Henkin theory:

\begin{lema} [The Maximum Principle] \label{MPL}
Let \mA\ be a complete Boolean algebra. Then, for every formula $\varphi(x)$ in  $\mathcal{L}_p(\tmA)$, there is $u\in {\bf V}^{\tmA}$ such that
$$\termvalue{\exists x\varphi(x)}_1 = \termvalue{\varphi(u)}_1.$$
In particular, if ${\bf V}^{\tmA} \models \exists x\varphi(x)$ then ${\bf V}^{\tmA} \models \varphi(u)$ for some  $u\in {\bf V}^{\tmA}$.
\end{lema}
\begin{proof}
The proof is obtained by a straightforward adaptation of the proof of~\cite[Lemma~1.27]{bell}.  The collection $X=\{ \termvalue{\varphi(u)} \ : \ u \in {\bf V}^{\tmA}\}$ is  a set, since \tmA\ is a set. By the Axiom of Choice, there is an ordinal $\alpha$ and a set $\{u_\xi \ : \xi < \alpha\} \subseteq {\bf V}^{\tmA}\}$ such that $X= \{ \termvalue{\varphi(u_\xi)} \ : \ \xi < \alpha\}$, hence $\termvalue{\exists x\varphi(x)}_1 = \bigvee_{\xi < \alpha}\termvalue{\varphi(u_\xi)}_1$. For each $\xi < \alpha$ let $a_\xi= \termvalue{\varphi(u_\xi)}_1 \land \sneg \bigvee_{\eta  < \xi}\termvalue{\varphi(u_\eta)}_1$, and let $u=\sum_{\xi < \alpha} a_\xi \odot u_\xi$. By the Mixing Lemma~\ref{mix-lem} and by Theorem~\ref{basicVB} items~(ii) and~(vii) it follows that $\termvalue{\exists x\varphi(x)}_1 = \termvalue{\varphi(u)}_1$.
\end{proof}

\begin{coro} \label{coroMPL}
Let $\varphi(x)$ be a formula in  $\mathcal{L}_p(\tmA)$ such that ${\bf V}^{\tmA} \models \exists x\varphi(x)$. Then:\\[1mm]
(i) For any $v \in {\bf V}^{\tmA}$ there exists $u \in {\bf V}^{\tmA}$ such that $\termvalue{\varphi(u)}_1=1$ and $\termvalue{\varphi(v)}_1=\termvalue{u\approx v}_1$.\\[1mm]
(ii) Let $\psi(x)$ be a formula in  $\mathcal{L}_p(\tmA)$ such that ${\bf V}^{\tmA} \models \varphi(u)$  implies that ${\bf V}^{\tmA} \models \psi(u)$, for every $u\in {\bf V}^{\tmA}$. Then   ${\bf V}^{\tmA} \models \forall x (\varphi(x) \to \psi(x))$.
\end{coro}
\begin{proof}
Is an easy adaptation of the proof of~\cite[Corollary~1.28]{bell}, taking into account Lemma~\ref{MPL} and Theorem~\ref{basicVB} items~(ii) and~(vii).
\end{proof}

\

\noindent
The notion of {\em core} for a Boolean-valued set (see~\cite{bell})  can be easily adapted to twist-valued sets:

\begin{defi} \label{defcore}
Let  $u \in {\bf V}^{\tmA}$. A {\em core} for $u$ is a set $v \subseteq {\bf V}^{\tmA}$ such that:~(i)~$\termvalue{x\pert u}_1=1$ for every $x \in v$; and~(ii)~for every $y \in {\bf V}^{\tmA}$ such that $\termvalue{y\pert u}_1=1$, there is a unique $x \in v$ such that $\termvalue{x\approx y}_1=1$.
\end{defi}

\begin{lema} \label{lemmacore}
Any  $u \in {\bf V}^{\tmA}$ has a core.
\end{lema}
\begin{proof}
Is an easy adaptation of the proof of~\cite[Lemma~1.31]{bell}.
\end{proof}

\

\noindent Let $\varnothing$ be the empty element of ${\bf V}^{\tmA}$. As happens with Boolean-valued models, if $u \in {\bf V}^{\tmA}$ is such that ${\bf V}^{\tmA} \models \sneg(u \approx \varnothing)$ then, by the Maximum Principle, any core of $u$ is nonempty.

\begin{coro} \label{coro2MPL}
Let $u \in {\bf V}^{\tmA}$ such that ${\bf V}^{\tmA} \models \sneg(u \approx \varnothing)$, and let $v$ be a core for $u$. Then, for any $x \in {\bf V}^{\tmA}$ there exists $y \in v$ such that $\termvalue{x \approx y}_1=\termvalue{x\pert u}_1$.
\end{coro}
\begin{proof}
Is follows from Corollary~\ref{coroMPL}.
\end{proof}

\

\noindent
From the results obtained above, one of the main results of the paper can be established:

\begin{teo} \label{modelZFC}
All the axioms (hence all the theorems) of \ZFC, when restricted to pure \ZF-languages $\mathcal{L}_p(\tmA)$ (recall Definition~\ref{def-langs}),  are valid in  ${\bf V}^{\tmA}$, for every \mA.
\end{teo}
\begin{proof}
It is a relatively easy (but arduous) adaptation of the proof of~\cite[Theorem~1.33]{bell}, taking into account the auxiliary results obtained within this section, which are similar to the ones required in~\cite{bell}. 
\end{proof}

\section{Twist-valued models for $(\pst,\neg)$} \label{genPS3}

In this section the three-valued model for set theory introduced by L\"owe and Tarafder in~\cite{low:tar:15}  will be extended to a class of twist-valued models.

As observed in Section~\ref{twistm}, the three-valued logic $(\pst,\neg)$ (denoted as $(\pst,\ast)$ in~\cite{low:tar:15}) was already considered in~\cite{ConSil:14} under the name  \mpt. Indeed, this logic has been independenly proposed by different authors at  several times, and with different motivations.\footnote{As mentioned in Section~\ref{Intro-twist}, $\lfium_\circ$ is another presentation of this logic.}  For instance, the same logic was proposed in 1970 by  da Costa and D'Ottaviano's as \dacdot. It was reintroduced in 2000 by Carnielli, Marcos and de Amo as \lfium\,  and by Batens and De Clerq as the propositional fragment of the first-order logic {\bf CLuNs}, in  2014. As observed by Batens,  this logic was firstly proposed by Karl Sc\"utte in 1960 under the name $\Phi_v$  (see \cite{CC16} for  details and specific references).
Each of the three-valued algebras above is equivalent, up to language, to the   three-valued algebra of \L ukasiewicz three-valued logic \L$_3$. Hence, these logics are equivalent to \L$_3$ with $\{{\bf 1},{\bf \frac{1}{2}}\}$ as designated values. Moreover, as it was shown by Blok and Pigozzi in~\cite{blok-pig}, the class of algebraic models of \dacdot\ (and so  the class of twist structures for  \lptz) coincides with the agebraic models of \L ukasiewicz's three-valued logic $\L_3$.  More remarks about these three-valued equivalent logics can be found in~\cite{CC16}, Chapters~4 and~7.

As shown in~\cite[p.~407]{ConSil:14}, the implication $\Ra$ given by

\begin{center}
\begin{tabular}{|c||c|c|c|}
\hline
$\Ra$ & {\bf 1}  & ${\bf \frac{1}{2}}$  & {\bf 0} \\
 \hline \hline
     {\bf 1}    & {\bf 1}  & {\bf 1}  & {\bf 0}   \\ \hline
     ${\bf \frac{1}{2}}$  & {\bf 1}  & {\bf 1}  & {\bf 0}   \\ \hline
     {\bf 0}    & {\bf 1}  & {\bf 1}  & {\bf 1}   \\ \hline
\end{tabular}
\end{center}

\

\noindent
(which is the same implication  $\Ra$ of \pst\ and the primitive implication of \mpt) can be defined in the language of \lfium\ (hence in the language of \lptz) as follows: $\varphi \Ra \psi=_{def} \neg\sneg(\varphi \to \psi)$. From this, it is easy to adapt Definition~\ref{defKlfi1} of twist-structures for \lptz\ to  $(\pst,\neg)$ (see Definition~\ref{defKps3} below). Hence, the logic $(\pst,\neg)$ will be considered as defined over the signature $\Sigma_\Ra=\{\land,\lor,\Ra, \neg\}$. As observed in~\cite[pp.~395 and~407]{ConSil:14}, the strong negation $\sneg$ can be defined as $\sneg\varphi=_{def} \varphi \Ra \neg(\varphi \Ra \varphi)$, while  $\varphi \to \psi=_{def} \sneg\varphi \vee \psi$.

\begin{defi}  \label{defKps3}
Let \mA\ be a complete Boolean algebra, and let \tA\ as in Definition~\ref{twdom}. The  {\em twist structure for $(\pst,\neg)$ over \mA} is the algebra  $\tmzA=\langle \tA, \tilde{\wedge}, \tilde{\vee}, \tilde{\Ra},\tilde{\wneg} \rangle$ over $\Sigma_\Ra$  such that the operations $\tilde{\wedge}$, $\tilde{\vee}$ and $\tilde{\wneg}$ are defined as in Definition~\ref{defKlfi1}, and $\tilde{\Ra}$ is defined as follows, for every $(z_1,z_2),(w_1,w_2) \in \tA$:

\begin{itemize}
  \item[]  $(z_1,z_2)\,\tilde{\Ra}\, (w_1,w_2) = (z_1 \imp w_1,z_1 \wedge \sneg w_1)$.
\end{itemize}
\end{defi}

\

\noindent
By considering (as mentioned above) ${\sim}$ and ${\to}$ as derived connectives in \tmzA, it is clear that $\tilde{{\sim}}(z_1,z_2) = ({\sim}z_1,z_1)$  and $(z_1,z_2)\,\tilde{\to}\, (w_1,w_2) = (z_1 \imp w_1,z_1 \wedge w_2)$. Hence, the original operations of Definition~\ref{defKlfi1} can be recovered in \tmzA.

As it will be discussed below, we will adopt a technique different to the one used in~\cite{low:tar:15} in order to show the satisfaction of \ZFC\ in the twist-valued models based on \tmzA. However, it is interesting to observe that a nice property of $(\pst,\neg)$ is preserved by any \tmzA. Indeed, in~\cite{low:tar:15} the following notion of {\em reasonable implication algebras} was proposed in order to provide suitable lattice-valued for \ZF: 

\begin{defi} \label{reasimp} An algebra $\mA=\langle A,\land,\lor,\Ra,{\bf 0},{\bf 1}\rangle$ is an {\em reasonable implication algebra} if the reduct $\langle A,\land,\lor,{\bf 0},{\bf 1}\rangle$  is a complete lattice with bottom {\bf 0} and top {\bf 1}, and $\Ra$ is a binary operator satisfying the following, for every $z,w,u \in A$:

\begin{itemize}
\item[(P1)] \ \ $z \land w \leq u$ implies that $z \leq (w \Ra u)$;
\item[(P2)] \ \ $z \leq w$ implies that $(u \Ra z) \leq (u  \Ra w)$;
\item[(P3)] \ \ $z \leq w$ implies that $(w \Ra u) \leq (z  \Ra u)$.
\end{itemize} 
\end{defi}

\begin{prop}
For every complete Boolean algebra \mA, the twist structure \tmzA\ for  $(\pst,\neg)$ is a reasonable implication algebra such that ${\bf 0}=(0,1)$ and ${\bf 1}=(1,0)$.\footnote{To be rigorous, the $\neg$-less reduct of \tmzA\ expanded with {\bf 0} and {\bf 1} is a reasonable implication algebra.}
\end{prop}
\begin{proof}
Let $(z_1,z_2),(w_1,w_2), (u_1,u_2) \in \tA$.\\[1mm]
(P1): Assume that $(z_1,z_2) \,\tilde{\wedge}\, (w_1,w_2) \leq (u_1,u_2)$. That is, $(z_1 \wedge w_1,z_2 \vee w_2) \leq (u_1,u_2)$. Then $z_1 \wedge w_1 \leq u_1$ and $z_2 \vee w_2 \geq u_2$. From $z_1 \wedge w_1 \leq u_1$ it follows that $z_1 \leq w_1 \to u_1$. Besides, since $z_1 \vee z_2=1$ then $\sneg z_2 \leq z_1 \leq w_1 \to u_1$. Hence $z_2 \geq \sneg(w_1 \to u_1)= w_1 \wedge \sneg u_1$. From this, $(z_1,z_2) \leq (w_1 \to u_1, w_1 \wedge \sneg u_1)=(w_1,w_2) \,\tilde{\Ra}\, (u_1,u_2)$.\\[1mm]
(P2): Assume that $(z_1,z_2) \leq (w_1,w_2)$. Then $z_1 \leq w_1$, hence $u_1 \to z_1 \leq u_1 \to w_1$ and so $u_1 \land \sneg z_1=\sneg(u_1 \to z_1) \geq  \sneg(u_1 \to w_1)=u_1 \land \sneg w_1$. This means that  $(u_1,u_2) \,\tilde{\Ra}\, (z_1,z_2) \leq (u_1,u_2) \,\tilde{\Ra}\,  (w_1,w_2)$.\\[1mm]
(P3): It is proved analogously, but now taking into account that  $z_1 \leq w_1$ implies that $w_1 \to u_1 \leq z_1 \to u_1$.
\end{proof}

\

\noindent
Now, the three-valued model of set theory presented in~\cite{low:tar:15}  will be generalized to twist-valued models over any complete Boolean algebra. The structure  ${\bf V}^{\tmzA}$ is defined as the structure  ${\bf V}^{\tmA}$ given in Definition~\ref{vTAdef}. This does not come as a surprise, given that the domain of \tmA\ and \tmzA\ is the same, the set \tA. However, ${\bf V}^{\tmA}$  and ${\bf V}^{\tmzA}$ are different  as first-order structures, namely, the way in which the formulas are interpreted. The only difference, besides using different implications in the underlying logics, will be in the form in which the predicates $\pert$ and $\approx$ are interpreted. Thus, the twist truth-value $\termvalue{\varphi}^{{\bf V}^{\tmzA}}$ of a sentence $\varphi$ in ${\bf V}^{\tmzA}$ will be defined according to the recursive clauses in  Definition~\ref{defint-twist}, with the following difference: any occurrence of the operator $\tilde{\to}$ must be replaced by the operator $\tilde{\Ra}$ Note that the clause  interpreting $\sneg\varphi$ is now derived from the others, taking into account the observation after Definition~\ref{defKps3}.

In Theorem~\ref{modelZFCN} below it is stated that every twist-valued structure ${\bf V}^{\tmzA}$ is a model of \ZFC. This constitutes a generalization of~\cite[Corollary~11]{low:tar:15}. Indeed, instead of taking just a three-valued model (generated by the two-element Boolean algebra), we obtain a class of models, one for each complete Boolean algebra. Moreover, we also prove that these generalized models (including, of course, the original L\"owe-Tarafder model) satisfy, in addition, the Axiom of Choice.  

The proof of validity of \ZF\ given in~\cite[Corollary~11]{low:tar:15} is strongly based on the particularities of the three-valued algebra of $(\pst,\neg)$.\footnote{For instance, the fact that expressions like $\termvalue{u\approx v} \Ra \termvalue{u\pert w}$ can only take either the value {\bf 0} or {\bf 1} is used several times in~\cite{low:tar:15}. Observe that, in \tmzA, the value of $z \, \tilde{\Ra} \, w$ is always of the form $(a,\sneg a)$ for some $a \in |\mA|$. Hence $\termvalue{u\approx v}^{{\bf V}^{\tmzA}}$ is always of the form $(a,\sneg a)$ for some $a \in |\mA|$.} This forces us to adapt, to this setting, the proof for twist-valued models over \tmA\ given in the previous sections (which, by its turn, is adapted from the proof for Boolean-valued sets).
Such adaptations from \tmA\ to \tmzA\ are immediate, and all the results and definitions proposed in the previous sections work fine for \tmzA. Hence, we obtain the second main result of the paper:

\begin{teo} \label{modelZFCN}
All the axioms (hence all the theorems) of \ZFC, when restricted to pure \ZF-languages $\mathcal{L}_p(\tmA)$, are valid in  ${\bf V}^{\tmzA}$, for every \mA.
\end{teo}

\begin{obs}
Oberve that, in~\cite[Corollary~11]{low:tar:15}, it was proved that \pst\ is a model of \ZF, not of \ZFC. Thus, Theorem~\ref{modelZFCN} improves the above mentioned result in two ways: it is generalized to arbitary Boolean algebras and, in addition, it proves that  the Axiom of Choice AC is also satisfied by all that models, including the original three-valued structure \pst.
\end{obs}

\section{ \ZFlp\ as a paraconsistent set theory} \label{sect-ZFparacon}

After proving that the two classes of twist-valued models proposed here are models of \ZFC, in this section the paraconsistent character of both classes of  models will be investigated. It will be shown that twist-valued models over \tmA\ (that is, over the logic \lptz) are ``more paraconsistent'' that the ones over \tmzA\ (that is, defined over  $(\pst,\neg)$).

Recall from Theorem~\ref{basicVB}(i) that $\termvalue{u\approx u} \in D_\mA$ for every $u$ in every twist-valued model  ${\bf V}^{\tmA}$. The interesting fact of \ZFlp\  is that it allows ``inconsistent'' sets, that is, elements of  ${\bf V}^{\tmA}$ such that the value of $(u \not\approx u)$ is also designated. Observe that ${\bf 1}=(1,0)$, ${\bf \frac{1}{2}}=(1,1)$ and ${\bf 0}=(0,1)$ are defined in every \tmA.  Since $z \in D_\mA$ iff $z=(1,a)$ for some $a \in A$ it follows that   ${\bf \frac{1}{2}} \leq z$ for every $z \in D_\mA$ (recalling the partial order for $\tmA$  considered in Remark~\ref{TA-lat}).
 
\begin{prop} \label{u-incons}
There exists $u \in {\bf V}^{\tmA}$ such that $\termvalue{u\approx u}={\bf \frac{1}{2}}$.
\end{prop}
\begin{proof}
Let $w$ be any element of  ${\bf V}^{\tmA}$, and let $u=\{\langle w,{\bf \frac{1}{2}}\rangle\}$. Since $\termvalue{w\approx w}  \in D_\mA$ then 
$\termvalue{w\pert u}=u(w) \,\tilde{\wedge}\, \termvalue{w\approx w} = {\bf \frac{1}{2}}\,\tilde{\wedge}\, \termvalue{w\approx w}={\bf \frac{1}{2}}$.
From this,
$\termvalue{u\approx u}=u(w) \,\tilde{\to}\, \termvalue{w\pert u}= {\bf \frac{1}{2}}\,\tilde{\to}\,  {\bf \frac{1}{2}} =  {\bf \frac{1}{2}}$.
\end{proof}

\

\noindent
From the last result it can be proven that \ZFlp\ is strongly paraconsistent, in the sense that there is a contradiction which is valid in the logic:

\begin{coro} Let $\sigma=\forall x (x\approx x)$. Then ${\bf V}^{\tmA} \models \sigma \land \neg\sigma$.
\end{coro}
\begin{proof}
Let ${\bf V}^{\tmA}$ be a twist-valued model  for \ZFlp. As observed above,  ${\bf \frac{1}{2}} \leq z$ for every $z \in D_\mA$. By Theorem~\ref{basicVB}(i), $\termvalue{v\approx v} \in D_\mA$ for every $v$ in ${\bf V}^{\tmA}$ and so ${\bf \frac{1}{2}} \leq \termvalue{v\approx v}$ for every $v$, that is, 
${\bf \frac{1}{2}} \leq \termvalue{\forall x(x\approx x)}$, by Definition~\ref{defint-twist}. On the other hand, $\termvalue{\forall x(x\approx x)} \leq \termvalue{u\approx u} = {\bf \frac{1}{2}}$ for $u$ as in Proposition~\ref{u-incons}. This shows that $\termvalue{\sigma}=\termvalue{\forall x(x\approx x)}={\bf \frac{1}{2}}$ and so $\termvalue{\neg\sigma}=\tilde{\neg}\,\termvalue{\sigma}={\bf \frac{1}{2}}$. Hence $\termvalue{\sigma \wedge \neg\sigma}=\termvalue{\sigma} \,\tilde{\wedge}\,\termvalue{\neg\sigma} = {\bf \frac{1}{2}}$, a designated value.
\end{proof}

\

\noindent
Since  the extensionality axiom of \ZF\ is satisfied by every twist-valued model ${\bf V}^{\tmA}$  for \ZFlp, $\termvalue{u\approx v} \in  D_\mA$ iff $u$ and $v$ have the same elements, that is: for every $w$ in ${\bf V}^{\tmA}$, $\termvalue{w\pert u} \in  D_\mA$ iff $\termvalue{w\pert v} \in  D_\mA$. However, nothing guarantees that $u$ and $v$ will have the same `non-elements', namely: it could be possible that  $\termvalue{\neg(w\pert u)} \in  D_\mA$ but $\termvalue{\neg(w\pert v)} \notin  D_\mA$, for some $w$ in ${\bf V}^{\tmA}$, even when $\termvalue{u\approx v} \in  D_\mA$. Given such $w$, consider the property $\varphi(x):=\neg(w \pert x)$, meaning that ``$w$ is a non-element of $x$''. Then, this  situation shows that ${\bf V}^{\tmA} \not\models ((u \approx v) \wedge \varphi(u)) \to \varphi(v)$, which constitutes a violation of the  Leibniz rule for the equality predicate $\approx$ in \ZFlp.

\begin{teo} \label{fail-leibniz}
The formula $\varphi(x):=\neg(w \pert x)$ is such that the Leibniz rule fails for it in every ${\bf V}^{\tmA}$, namely:
${\bf V}^{\tmA}\not\models \forall x\forall y((x\approx y) \land \varphi(x) \to \varphi(y))$.
\end{teo}
\begin{proof}
Let ${\bf V}^{\tmA}$ be a twist-valued model  for \ZFlp, and let $\varnothing$ be the empty element of ${\bf V}^{\tmA}$. Observe that  $w=\{\langle\varnothing,{\bf 1}\rangle\}$, $u=\{\langle w,{\bf \frac{1}{2}}\rangle\}$ and $v=\{\langle w,{\bf 1}\rangle\}$ belong to every model ${\bf V}^{\tmA}$. Now, $\termvalue{\varnothing\pert w} = w(\varnothing)\,\tilde{\wedge}\,\termvalue{\varnothing\approx \varnothing} = {\bf 1} \,\tilde{\wedge}\,{\bf 1}={\bf 1}$. From this, $\termvalue{w\approx w} = w(\varnothing) \,\tilde{\to}\, \termvalue{\varnothing\pert  w} = {\bf 1} \,\tilde{\to}\, {\bf 1}={\bf 1}$ and so $\termvalue{w\pert  u} = u(w) \,\tilde{\wedge}\, \termvalue{w\approx w} = {\bf \frac{1}{2}} \,\tilde{\wedge}\, {\bf 1}= {\bf \frac{1}{2}}$.  On the other hand, $\termvalue{w\pert  v} = v(w) \,\tilde{\wedge}\, \termvalue{w\approx w} = {\bf 1} \,\tilde{\wedge}\, {\bf 1}= {\bf 1}$. This implies that $\termvalue{u\approx v} = (u(w) \,\tilde{\to}\, \termvalue{w\pert  v}) \,\tilde{\wedge}\,(v(w) \,\tilde{\to}\, \termvalue{w\pert  u}) = ({\bf \frac{1}{2}} \,\tilde{\to}\, {\bf 1}) \,\tilde{\wedge}\,({\bf 1} \,\tilde{\to}\, {\bf \frac{1}{2}}) = {\bf \frac{1}{2}}$.

But $\termvalue{\varphi(u)}=\termvalue{\neg(w\pert  u)} = \tilde{\neg}\, \termvalue{w\pert  u} = \tilde{\neg}\, {\bf \frac{1}{2}}={\bf \frac{1}{2}}$ and $\termvalue{\varphi(v)}=\termvalue{\neg(w\pert  v)} = \tilde{\neg}\, \termvalue{w\pert  v} = \tilde{\neg}\, {\bf 1}={\bf 0}$. Thus,  $\termvalue{((u \approx v) \wedge \varphi(u)) \to \varphi(v)}=({\bf \frac{1}{2}} \,\tilde{\wedge}\, {\bf \frac{1}{2}})\,\tilde{\to}\, {\bf 0}={\bf 0}$, which implies that ${\bf V}^{\tmA}\not\models \forall x\forall y((x\approx y) \land \varphi(x) \to \varphi(y))$.
\end{proof} 

\

\noindent It is important to observe that the failure of the Leiniz rule in ${\bf V}^{\tmA}$ shown in Theorem~\ref{fail-leibniz} does not contradict Theorem~\ref{basicVB}(viii): indeed, what Theorem~\ref{basicVB}(viii) states is the validity of the Leibniz rule in ${\bf V}^{\tmA}$ for every formula $\varphi(x)$ in the pure \ZF-language $\mathcal{L}_p(\tmA)$. On the other hand, the formula  $\varphi(x)$ found in Theorem~\ref{fail-leibniz} which violates the Leibniz rule in ${\bf V}^{\tmA}$  contains an occurrence of the paraconsistent negation $\neg$, that is, it does not belong to $\mathcal{L}_p(\tmA)$. In that example, two sets which are equal have different `non-elements', where `non'  refers to the paraconsistent negation $\neg$. 

Besides the failure of the Leibniz rule for the full language, \ZFlp\ does not validate the so-called bounded quantification properties. 

\begin{defi} For any formula $\varphi$ and every $u \in {\bf V}^{\tmA}$, the {\em universal bounded quantification property} $UBQ_\varphi^u$ and the {\em existential bounded quantification property} $EBQ_\varphi^u$ are defined as follows:

\begin{itemize}
\item[] $(UBQ_\varphi^u)$ \ \ 
$\termvalue{\forall x(x \pert u \to \varphi(x))}_1=\bigwedge_{x \in dom(u)} ((u(x))_1 \to \varphi(x))$
\item[] $(EBQ_\varphi^u)$ \ \  $\termvalue{\exists x (x \pert u \, \wedge \, \varphi(x))}_1 = \bigvee_{x \in dom(u)} ((u(x))_1 \wedge \termvalue{\varphi(x)}_1)$
\end{itemize}
\end{defi}

By simplicity, formulas on the left-hand size of $UBQ_\psi^u$ and  $EBQ_\varphi^u$ will be written as $\termvalue{\forall x \pert u \, \varphi(x)}_1$ and $\termvalue{\exists x \pert u \, \varphi(x)}_1$, respectively.

By adapting the proof of~\cite[Corollary~1.18]{bell} it can be proven the following:

\begin{teo}
For any negation-free formula $\varphi$ (i.e., $\varphi \in \mathcal{L}_p(\tmA)$) and every $u \in {\bf V}^{\tmA}$, the bounded quantification properties  $UBQ_\varphi^u$ and  $EBQ_\varphi^u$   hold in ${\bf V}^{\tmA}$. 
\end{teo}

However, for formulas containing the paraconsistent negation the latter result does not holds in general:

\begin{prop} There is  $u \in {\bf V}^{\tmA}$ and formulas $\varphi(x)$ and $\psi(x)$ such that the bounded quantification properties $UBQ_\psi^u$ and  $EBQ_\varphi^u$  fail in ${\bf V}^{\tmA}$.
\end{prop}
\begin{proof}
It is enough to prove the falure of  $EBQ_\varphi^u$ given that the failure of $UBQ_\psi^u$ is obtained from it by using $\psi(x):=\sneg\varphi(x)$ and the duality between infimum and supremum through the Boolean complement $\sneg$.

Thus, let ${\bf V}^{\tmA}$  and let $w=\{\langle\varnothing,{\bf 1}\rangle\}$, $v=\{\langle w,{\bf \frac{1}{2}}\rangle\}$,  $y=\{\langle w,{\bf 1}\rangle\}$ and $u=\{\langle y,{\bf 1}\rangle\}$.  Let $\varphi(x):=\neg(w \pert x)$.   As in the proof of Theorem~\ref{fail-leibniz} it can be proven that $\termvalue{v \approx y}=\termvalue{\varphi(v)}={\bf \frac{1}{2}}$ and $\termvalue{\varphi(y)} = {\bf 0}$.
Hence $\bigvee_{x \in dom(u)} ((u(x))_1 \wedge \termvalue{\varphi(x)}_1) =  (u(y))_1 \wedge \termvalue{\varphi(y)}_1 = 0$ while $\termvalue{\exists x \pert u \, \varphi(x)}_1 = \termvalue{\exists x (x \pert u \, \wedge \, \varphi(x))}_1 = \bigvee_{v' \in {\bf V}^{\tmA}} \bigvee_{x \in dom(u)} ((u(x))_1 \wedge \termvalue{v' \approx x}_1 \wedge  \termvalue{\varphi(v')}_1) =  \bigvee_{v' \in {\bf V}^{\tmA}}  ((u(y))_1 \wedge \termvalue{v' \approx y}_1 \wedge  \termvalue{\varphi(v')}_1) \geq (u(y))_1 \wedge \termvalue{v \approx y}_1 \wedge  \termvalue{\varphi(v)}_1 = 1$. This means that
$\termvalue{\exists x \pert u \, \varphi(x)}_1 = 1 \neq 0 =  \bigvee_{x \in dom(u)} ((u(x))_1 \wedge \termvalue{\varphi(x)}_1)$.
\end{proof}

\

\noindent It is worth noting that the limitations of \ZFlp\ pointed out above (namely, the Leibniz rule and the bounded quantification property for formulas containing the paraconsistent negation) are also present in  L\"owe-Tarafder's model~\cite{low:tar:15}.

As mentioned in Section~\ref{Intro-twist}, in~\cite{CC13} was presented a family of paraconsistent set theories based on diverse \lfis, such that the original \ZF\ axioms were slightly modified in order to deal with a unary predicate $C(x)$ representing that `the set $x$ is consistent'. The consistency connective $\cons$ is primitive in \mbc, but it is definable as $\cons \varphi:=\sneg(\varphi \land \neg\varphi)$ in any axiomatic extension of \mbc\ which proves the schema (ciw): $\cons\varphi \vee (\varphi \land \neg\varphi)$ such as \lptz. In the same way, the consistency predicate $C(x)$ can be expressed, in extensions of \ZFmbc, in terms of a formula of \ZFmbc\ without using the predicate $C$, and the same happens with the inconsistency predicate $\neg C(x)$. For instance,  \ZFmci\ is based on \mci, an extension of \mbc\ in which $\neg\cons\varphi$ is equivalent to $\varphi \land \neg\varphi$. Thus, $\neg C(x)$ was defined to be equivalent to $(x \approx x) \land \neg(x \approx x)$   in  \ZFmci. From this,  $\neg C(x)$ is equivalent to $\neg\cons(x \approx x)$   in  \ZFmci. Given that \lptz\ is an extension of \mci, if a consistency predicate for sets were added to the language of \ZFlp\ then it seems reasonable to require the equivalence between $\neg C(x)$ and $\neg\cons(x \approx x)$   in  \ZFlp. But $\cons C(x)$ is derivable \ZFmci, so it would be valid in  \ZFlp\ (indeed, the proof in \ZFmci\ of $\cons C(x)$ given in~\cite[Proposition~3.10]{CC13} holds in \qlptz, assuming the axioms for $C$ from \ZFmci). From this  $C(x) \sse \cons(x \approx x)$ would be also derivable in \qlptz\ and so it would be valid in \ZFlp\ expanded with a suitable predicate $C$ denoting `consistency for sets'. This motivates the following:  

\begin{defi} \label{conssets} Define in \ZFlp\ the {\em consistency  predicate for sets}, $C(x)$, as follows: $C(x)=_{def} \sneg\neg(x\approx x)$.
\end{defi}

\noindent According to the previous discussion,  $C(x)$ should be equivalent to $\cons(x\approx x)$ in \ZFlp. But $\cons\varphi$ is equivalent to $\sneg(\varphi \land \neg \varphi)$ in \lptz, and $(x\approx x)$ is valid in \ZFlp, hence
$C(x)$ should be equivalent to $\sneg\neg(x\approx x)$ in \ZFlp, which justifies  Definition~\ref{conssets}.

\begin{prop} The consistency predicate $C(x)$ is non-trivial: there exist $v,w \in {\bf V}^{\tmA}$ such that $\termvalue{C(v)}={\bf 1}$ and $\termvalue{C(w)}={\bf 0}$. Moreover,  $\termvalue{C(u)}\neq {\bf \frac{1}{2}}$ for every $u$ in ${\bf V}^{\tmA}$.
\end{prop}
\begin{proof}
Let ${\bf V}^{\tmA}$ be a twist-valued model for \ZFlp, and consider $v=\{\langle\varnothing,{\bf 1}\rangle\}$ and $w=\{\langle\varnothing,{\bf \frac{1}{2}}\rangle\}$ in ${\bf V}^{\tmA}$. It is easy to see that $\termvalue{C(v)}={\bf 1}$ and $\termvalue{C(w)}={\bf 0}$. On the other hand, for every $u$ in ${\bf V}^{\tmA}$ it is the case that  $\termvalue{C(u)}=\tilde{\sneg} z$ for $z = \termvalue{\neg(u \approx u)}$. Hence $\termvalue{C(u)} = (\sneg z_1,z_1)\neq {\bf \frac{1}{2}}$, for every $u$.
\end{proof}

Finally, we can show now that twist-valued models over \tmA\ (that is, over the logic \lptz) are ``more paraconsistent'' than the ones over \tmzA\ (that is, defined over  $(\pst,\neg)$). Indeed, as we have seen,  \ZFlp\  allow us to define in every twist-valued model  ${\bf V}^{\tmA}$  an ``inconsistent set'', namely  $u$, such that $(u \approx u) \land \neg(u \approx u)$ holds. In fact, any $u = \{\langle w, {\bf \frac{1}{2}}\rangle\}$ is such that $\termvalue{u \approx u}={\bf \frac{1}{2}} \tilde{\to}{\bf \frac{1}{2}}={\bf \frac{1}{2}}$.  The difference, of course, rests on the nature of the  implication operator considered in each case: in $(\pst,\neg)$ the value of   $(u \approx u)$ is always {\bf 1}, since ${\bf \frac{1}{2}} \tilde{\Rightarrow}{\bf \frac{1}{2}}={\bf 1}$. Hence,   $\neg(u \approx u)$ always gets the value {\bf 0}. The same holds  in any model over reasonable implicative algebras considered by L\"owe and Tarafder (see~\cite[Proposition~1]{low:tar:15}).

\subsection{Discussion: \ZFlp\ and the failure of the Leibniz rule}

At first sigth, having a (paraconsistent) set theory as \ZFlp\ in which the Leibniz rule is not  satisfied for every formula $\varphi(x)$ that represents a property could seem to be a bit disappointing. After all, \ZF\ is defined as a first-order theory with equality, which pressuposes the validity of the Leibniz rule.  

The Leibniz rule states that the equality predicate preserves logical equivalence, namely: $(a \approx b) \to (\varphi(a) \sse \varphi(b)$ for every formula $\varphi(x)$ (clearly this can be generalized to formulas with $n\geq 1$ free variables, assuming $\bigwedge_{i=1}^n(a_i \approx b_i)$). In first-order theories based on classical logic, such as \ZF, it is enough to require that this property holds for every atomic formula, and so the general case is proven by induction on the complexity of $\varphi$. Of course this proof cannot be reproduced in \qlptz, since $\neg$ is not congruential: $\varphi(a) \sse \varphi(b)$ does not imply $\neg\varphi(a) \sse \neg\varphi(b)$ in general (and this is the key step in the proof by induction). The solution is requiring the validity of the Leibniz rule for every $\varphi$ from the beginning, adjusting accordingly the class of interpretations for  \qlptz\ expanded with equality  (see~\cite{CFG19}). However, the situation for  \ZFlp\ is quite different: because of the extensionality axiom, the definition of the interpretation of the equality predicate depends strongly on the interpretation of the membership predicate. In fact, the interpretation of both predicates is simultaneously defined by transfinite recursion, according to Definition~\ref{defint-twist}.

The validity of the Leibniz rule, in the case of Boolean-set models for \ZFC, is  proven as a theorem.  The simultaneous definition of the equality and membership predicates is designed to fit exactly the requirements of the extensionality axiom: two individuals (sets) are identical provided that they have the same elements. From this, it is proven by induction of the complexity of $\varphi(x)$ that $\termvalue{u \approx v} \land \termvalue{\varphi(u)} \leq \termvalue{\varphi(v)}$ in every Boolean-valued model. As we have seen in Theorem~\ref{basicVB}(vii), the same holds in twist-valued models w.r.t. the first coordinate, namely: $\termvalue{u \approx v}_1 \land \termvalue{\varphi(u)}_1 \leq \termvalue{\varphi(v)}_1$. But then, it is required that this property  just holds for `classical'  formulas, that is, formulas $\varphi$ without occurrences of the paraconsistent negation $\neg$. The explanation for this fact is simple, from the technical point of view: assuming that the property above holds for $\varphi$ then, when considering $\neg\varphi$, the value of $\termvalue{\neg\varphi(u)}_1$ is $\termvalue{\varphi(u)}_2$, and we don't have enough information about the relationship between $\termvalue{\varphi(u)}_2$, and $\termvalue{\varphi(v)}_2$. The example given in the proof of  Theorem~\ref{fail-leibniz} shows that it is impossible to satisfy the Leibniz rule in   \ZFlp\ for formulas containing the paraconsistent negation, hence this is an unsolvable problem with the current definitions.


Within the present approach, paraconsistent situations such as the existence of `inconsistent'  sets $u$ satisfying $\neg(u \approx u)$ or the existence of a set being simultaneously an element and a non-element of another set seems to be  irreconcilable with the fullfillment of the Leibniz rule for formulas behind the `classical' language. Because of this, the predicate $\approx$ in \ZFlp\ 
should be considered as representing `indiscernibility by pure \ZF-properties', exactly as happens with Boolean-valued models for \ZF. In this manner $(u  \approx v)$ implies that, besides having the same elements, $u$ and $v$ have,  for instance, the same `non$^*$-elements', where `non$^*$' stands for the classical negation $\sneg$. That is, $\forall w(\sneg(w \pert u) \sse \sneg (w \pert v))$ is a consequence of $(u  \approx v)$. On the other hand, as it was shown in Theorem~\ref{fail-leibniz}, $(u  \approx v)$ {\em does not imply} (in general) that $u$ and $v$ have the same `non-elements', where `non' stands for the paraconsistent negation $\neg$:  $\forall w(\neg(w \pert u) \sse \neg (w \pert v))$ is not a consequence of $(u  \approx v)$.

Instead of being regarded as discouraging,  the fact that  $(u  \approx v)$  does not  necessarily  imply   that $u$ and $v$ have the same `non-elements' (for `non'  the paraconsistent negation $\neg$) can be seen as an auspicious   property, because it can be a  way to circumvent undesirable consequences of  `non-elements', as it happens with the well-known    Hempel's Ravens Paradox:   evidence, differently  from proof, for instance, has its  own   idiosyncratic properties. This  point, however, will be left for further discussion.

\section{Concluding remarks}

In this paper, we introduce a generalization of Boolean-valued models of set theory to a class of algebras represented as twist-structures, defining a class of models for \ZFC\ that we called twist-valued models.  This class of algebras characterizes a three-valued paraconsistent logic called \lpt, which was extensively studied in the literature of paraconsistent logics under different names and signatures as, for example, as the well-known  da Costa and D'Ottaviano's logic \dacdot\ and as  the logic  \lfium\ (cf. \cite{car:mar:amo:2000}) . As it was shown by Blok and Pigozzi in~\cite{blok-pig}, the class of algebraic models of \dacdot\  (hence, the class of twist structures for \lptz) coincides with the agebraic models of \L ukasiewicz three-valued logic $\L_3$.

With small changes, in Section~\ref{genPS3} the twist-valued models for \lptz\ were adapted in order to obtain twist-valued for  $(\pst,\neg)$, the three-valued paraconsistent logic studied  by L\"owe and Tarafder in~\cite{low:tar:15} as a basis for    paraconsistent set theory. Thus, their three-valued algebraic model of \ZF\ was extended to a class of twist-valued models of \ZF, each of them defined over a complete Boolean algebra. In addition, it was proved that these models (including the  three-valued model over   $(\pst,\neg)$)  satisfy, in addition, the Axiom of Choice.
Moreover, it was shown that the implication operator $\to$ of \lptz\ is, in a sense, more suitable for a paraconsistent set theory than the one $\Ra$ of \pst: it allows inconsistent sets (i.e., $\termvalue{(w \approx w)}={\bf \frac{1}{2}}$ for some $w$, see   Proposition~\ref{u-incons}). It is worth noting that $\to$ {\em does not} characterize a `reasonable implication algebra' (recall Definition~\ref{reasimp}): indeed, ${\bf 1} \wedge {\bf \frac{1}{2}} \leq {\bf \frac{1}{2}}$ but ${\bf 1} \not\leq {\bf \frac{1}{2}} \to  {\bf \frac{1}{2}} = {\bf \frac{1}{2}}$. This shows that reasonable implication algebras are just one way to define a paraconsistent set theory, not the best.
 
Despite having the same limitative results than  L\"owe-Tarafder's model (that is, the debatable failure of Leibniz rule and the bounded quantification property for formulas containing the paraconsistent negation, recall Section~\ref{sect-ZFparacon}) we believe  that \ZFlp\ has a great potential as a paraconsistent set theory. 
In particular, the formal properties and the axiomatization of \ZFlp\ deserve to be further investigated, especially towards the problem of the validity of independence results  in paraconsistent set theory.

\bibliographystyle{abbrv}

\end{document}